\ifdefined\pdfoutput
  \pdfoutput=1
\fi

\documentclass[11pt,a4paper]{article}

\usepackage[T1]{fontenc}
\usepackage[utf8]{inputenc}
\usepackage[english]{babel}
\usepackage{lmodern}
\usepackage[final,protrusion=true,expansion=true]{microtype}
\usepackage[
  top=2.1cm,
  bottom=2.1cm,
  left=2.3cm,
  right=2.3cm,
  footskip=28pt
]{geometry}

\usepackage{setspace}
\usepackage{amsmath,amssymb,amsfonts,amsthm,mathtools}
\usepackage{bm}
\usepackage{bbm}
\allowdisplaybreaks
\numberwithin{equation}{section}
\mathtoolsset{showonlyrefs=false}

\usepackage{xcolor}
\usepackage{graphicx}
\graphicspath{{figures/}{images/}}
\usepackage{booktabs}
\usepackage{array}
\usepackage{adjustbox}
\usepackage{caption}
\usepackage{subcaption}
\usepackage{float}
\usepackage[section]{placeins}

\usepackage{tikz}
\usetikzlibrary{arrows.meta,calc,decorations.markings,decorations.pathreplacing,positioning}
\tikzset{
  vtx/.style={circle,fill=black,draw=black,inner sep=2pt},
  gluevtx/.style={circle,fill=red,draw=red,inner sep=2.15pt},
  ed/.style={blue,line width=1.15pt},
  every picture/.style={line cap=round,line join=round}
}

\usepackage{algorithm}
\usepackage{algpseudocode}
\floatname{algorithm}{Algorithm}

\usepackage{enumitem}
\setlist[itemize]{leftmargin=2.2em,itemsep=0.15em,topsep=0.35em}
\setlist[enumerate]{leftmargin=2.2em,itemsep=0.15em,topsep=0.35em}

\usepackage{authblk}

\usepackage{titlesec}
\titleformat{\section}
  {\large\bfseries}{\thesection.}{0.65em}{}
\titleformat{\subsection}
  {\normalsize\bfseries}{\thesubsection.}{0.65em}{}
\titleformat{\subsubsection}
  {\normalsize\itshape}{\thesubsubsection.}{0.65em}{}
\titlespacing*{\section}{0pt}{2.4ex plus .4ex minus .2ex}{1.1ex}
\titlespacing*{\subsection}{0pt}{1.9ex plus .3ex minus .2ex}{0.8ex}
\titlespacing*{\subsubsection}{0pt}{1.5ex plus .2ex minus .2ex}{0.6ex}

\usepackage{aliascnt}

\theoremstyle{plain}
\newtheorem{theorem}{Theorem}[section]

\newaliascnt{lemma}{theorem}
\newtheorem{lemma}[lemma]{Lemma}
\aliascntresetthe{lemma}

\newaliascnt{proposition}{theorem}
\newtheorem{proposition}[proposition]{Proposition}
\aliascntresetthe{proposition}

\newaliascnt{corollary}{theorem}
\newtheorem{corollary}[corollary]{Corollary}
\aliascntresetthe{corollary}

\newaliascnt{conjecture}{theorem}

\aliascntresetthe{conjecture}

\newaliascnt{problem}{theorem}
\newtheorem{problem}[problem]{Problem}
\aliascntresetthe{problem}

\newaliascnt{observation}{theorem}

\aliascntresetthe{observation}

\newaliascnt{claim}{theorem}
\newtheorem{claim}[claim]{Claim}
\aliascntresetthe{claim}

\newaliascnt{case}{theorem}

\aliascntresetthe{case}

\theoremstyle{definition}
\newaliascnt{definition}{theorem}

\aliascntresetthe{definition}

\newaliascnt{example}{theorem}

\aliascntresetthe{example}

\theoremstyle{remark}
\newaliascnt{remark}{theorem}

\aliascntresetthe{remark}

\newaliascnt{question}{theorem}

\aliascntresetthe{question}

\newaliascnt{fact}{theorem}

\aliascntresetthe{fact}

\newcommand{\eps}{\ensuremath{\varepsilon}}

\DeclareMathOperator{\ex}{ex}

\usepackage[numbers,sort&compress]{natbib}
\definecolor{linkblue}{RGB}{0,70,150}
\usepackage[
  colorlinks=true,
  linkcolor=linkblue,
  citecolor=blue,
  urlcolor=black,
  bookmarksnumbered=true,
  bookmarksopen=true
]{hyperref}

\usepackage{bookmark}
\usepackage[nameinlink,noabbrev,capitalise]{cleveref}

\crefname{theorem}{Theorem}{Theorems}
\crefname{lemma}{Lemma}{Lemmas}
\crefname{proposition}{Proposition}{Propositions}
\crefname{corollary}{Corollary}{Corollaries}
\crefname{conjecture}{Conjecture}{Conjectures}
\crefname{problem}{Problem}{Problems}
\crefname{observation}{Observation}{Observations}
\crefname{claim}{Claim}{Claims}
\crefname{case}{Case}{Cases}
\crefname{definition}{Definition}{Definitions}
\crefname{example}{Example}{Examples}
\crefname{remark}{Remark}{Remarks}
\crefname{question}{Question}{Questions}
\crefname{fact}{Fact}{Facts}
\crefname{equation}{equation}{equations}
\crefname{figure}{Figure}{Figures}
\crefname{table}{Table}{Tables}
\crefname{algorithm}{Algorithm}{Algorithms}
\Crefname{equation}{Equation}{Equations}

\title{\bfseries Tight Bound for Nikiforov's Spectral Even-Cycle Conjecture}
\author[1]{Peiru Kuang\thanks{Email: \href{mailto:peiru\_k@sjtu.edu.cn}{peiru\_k@sjtu.edu.cn}. }}
\author[1]{Feng Liu\thanks{Email: \href{mailto:liufeng0609@126.com}{liufeng0609@126.com}.}}
\author[1]{Shuang Sun\thanks{Email: \href{mailto:chocolatesun@sjtu.edu.cn}{chocolatesun@sjtu.edu.cn}.}}
\author[1]{Yan Wang\thanks{Email: \href{mailto:yan.w@sjtu.edu.cn}{yan.w@sjtu.edu.cn}.}}
\author[2]{Jiasheng Zeng\thanks{Email: \href{mailto:jasonzeng@mail.ustc.edu.cn}{jasonzeng@mail.ustc.edu.cn}.}}

\affil[1]{School of Mathematical Sciences, Shanghai Jiao Tong University, Shanghai 200240, China}
\affil[2]{Department of Mathematics,
Hong Kong University of Science and Technology,
Clear Water Bay, Hong Kong}

\date{}

\begin{document}

\maketitle
\thispagestyle{plain}
\vspace{-1.2em}

\begin{abstract}
Nikiforov conjectured that, for every fixed $k\ge2$ and all sufficiently large $n$, the unique $n$-vertex $C_{2k+2}$-free graph with maximum adjacency spectral radius is $S^+_{n,k}$, where $S_{n,k}=K_k\vee\overline K_{n-k}$ and $S^+_{n,k}$ is obtained from $S_{n,k}$ by adding one edge inside the independent part. Cioab\u{a}, Desai and Tait proved this conjecture for $n\ge k^{O(k)}$. Later, Li and Ning raised the problem of determining the optimal exponent $\gamma=\gamma(k)$ such that the same conclusion holds for $n\ge \Omega(k^{\gamma(k)})$.

We prove a stronger uniform theorem for Nikiforov's matrices $A_\alpha(G)=\alpha D(G)+(1-\alpha)A(G)$. More precisely, for every $\eps>0$ there are constants $C_\eps$ and $k_\eps$ such that for all $0\le\alpha\le1-\eps$, $k\ge k_\eps$ and $n\ge C_\eps k$, every $n$-vertex $C_{2k+2}$-free graph $G$ satisfies $\rho_\alpha(G)\le\rho_\alpha(S^+_{n,k})$, with equality if and only if $G\cong S^+_{n,k}$. In particular, the case $\alpha=0$ answers the problem of Li and Ning, and the $A_\alpha$-spectral even-cycle threshold is linear in $k$, uniformly for all $\alpha$ bounded away from $1$.

Our proof introduces a weighted rooted Erd\H{o}s--Gallai type path lemma, which may be of independent interest in Perron-vector methods for spectral extremal graph problems. The same method also yields asymptotically tight $A_\alpha$-spectral bounds for two local forbidden-subgraph families, namely $(K_1\vee P_\ell)$-free graphs and $F_s$-free graphs, where $F_s$ denotes the friendship graph.
\end{abstract}

\section{Introduction}
The classical Tur\'an problem goes back to Mantel~\citep{Mantel1907} and Tur\'{a}n~\citep{Turan1941}. It asks for the maximum number of edges in an $n$-vertex graph that contains no copy of a fixed graph $H$. 
Its spectral analogue replaces the number of edges by the adjacency spectral radius: for a fixed forbidden graph $H$, one seeks the largest possible value of $\lambda(G)$ over all $n$-vertex $H$-free graphs $G$, where $\lambda(G)$ denotes the adjacency spectral radius of $G$. This viewpoint is natural since a large spectral radius often forces subgraphs to appear. This phenomenon was already present in early results of Nosal, Wilf and Stanley~\citep{Nosal1970,Wilf1986,Stanley1987}, which gave spectral conditions for the existence of certain subgraphs. Since then, many spectral Tur\'an results have been obtained, including results for paths, cycles and cliques \citep{Nikiforov2007,Nikiforov2008,Nikiforov2010}. Beyond these concrete forbidden subgraphs, there are also general results on spectral extremal graphs~\citep{ByrneDesaiTait2025}, as well as structural spectral extremal theorems for planar graphs, minor-closed classes, color-critical graphs, wheels and related families \citep{CioabaDesaiTait2022OddWheels,CioabaDesaiTait2023,DvorakMohar2010,TaitTobin2017,Tait2019CdV}.

To study spectral extremal problems beyond the adjacency matrix, Nikiforov~\citep{Nikiforov2017} introduced the matrices $A_\alpha(G)$, defined for $0\le \alpha\le 1$ by
$$ A_\alpha(G)=\alpha D(G)+(1-\alpha)A(G),$$
where $A(G)$ is the adjacency matrix and $D(G)$ is the degree matrix of $G$. Let $\rho_\alpha(G)$ denote the largest eigenvalue of $A_\alpha(G)$. This family connects the adjacency matrix with the degree matrix. In particular, when $\alpha=0$, we have $\rho_0(G)=\lambda(G)$, and when $\alpha=1/2$, the matrix $A_{1/2}(G)$ is one half of the signless Laplacian matrix. Thus the $A_\alpha$ setting allows one to study adjacency and signless-Laplacian spectral extremal problems in a unified way. Recent work in this direction includes the comparison theorem of Byrne, Desai and Tait~\citep{ByrneDesaiTait2025}, as well as $A_\alpha$-spectral extremal results for even cycles, linear forests, Erd\H{o}s--S\'os type problems and Erd\H{o}s--P\'osa type problems \citep{ChenLiLiYuZhang2023AalphaES,ChenLiuZhang2023LinearForestsAalpha,LiYu2023AalphaEvenCycles,LiYuZhang2023AalphaErdosPosa,NikiforovYuan2015QEvenCycles}.

In this paper we focus on the $A_\alpha$-spectral Tur\'an problem for even cycles. Even cycles are among the basic forbidden bipartite graphs, but they already show the main difficulty of sparse extremal problems. In ordinary extremal graph theory, bipartite forbidden graphs are not covered by the Erd\H{o}s--Stone--Simonovits theorem in the same way as non-bipartite forbidden graphs; see the survey of F\"uredi and Simonovits~\citep{FurediSimonovits2013}. Extremal problems for prescribed cycle lengths and for arithmetic progressions of cycle lengths were studied by Verstra\"ete~\citep{Verstraete2000,Verstraete2016}. In the spectral setting, Babai and Guiduli~\citep{BabaiGuiduli2009} proved a spectral Zarankiewicz theorem for bipartite forbidden graphs. Exact spectral results for even cycles require more detailed structure. A large spectral radius may be supported on a small set of high-weight vertices, and this concentration can force long even cycles even when the average degree is only linear in the forbidden length. The natural extremal candidate has this same concentrated form: a small clique joined to a large independent set.

For $n>k$, let $S_{n,k}=K_k\vee \overline K_{n-k}$, and let $S^+_{n,k}$ be obtained from $S_{n,k}$ by adding one edge inside the independent part. The graph $S^+_{n,k}$ is $C_{2k+2}$-free. Indeed, write $L$ for the clique part and $R$ for the independent part. Every cycle contains at most as many vertices from $R$ as from $L$ in $S_{n,k}$. Since $|L|=k$, every cycle in $S_{n,k}$ contains at most $k$ vertices from $R$. After adding one edge inside $R$, a cycle can contain at most one more vertex from $R$ than from $L$. Hence every cycle in $S^+_{n,k}$ has length at most $2k+1$, and so $S^+_{n,k}$ contains no copy of $C_{2k+2}$. Thus $S^+_{n,k}$ gives a natural lower-bound construction for the $C_{2k+2}$-free problem. Nikiforov conjectured that, for every fixed $k\ge2$ and all sufficiently large $n$, this construction is best possible in the adjacency spectral setting: $S^+_{n,k}$ is the unique $n$-vertex $C_{2k+2}$-free graph with maximum adjacency spectral radius. The first exact results were obtained for $k=2$, that is, for the forbidden-$C_6$ problem, by Zhai and Lin~\citep{ZhaiLin2020} and by Zhai, Wang and Fang~\citep{ZhaiWangFang2020}. Further progress on this family was made by Wang, Kang and Xue~\citep{WangKangXue2023}. The full conjecture was proved by Cioab\u{a}, Desai and Tait~\citep{CioabaDesaiTait2024}; their proof gives a sufficient range of the form $n\ge k^{O(k)}$.

Although this settles the conjecture for each fixed $k$, it does not determine the right dependence of $n$ on $k$. The range $n\ge k^{O(k)}$ leaves open a sharper threshold question: how large must $n$ be before the comparison with $S^+_{n,k}$ already forces the forbidden cycle? This question was raised by Li and Ning~\citep{LiNing2023} and was also recorded by Liu and Ning~\citep{liu2023unsolved}.

\begin{problem}[Li-Ning~\citep{LiNing2023}, and Liu-Ning~\citep{liu2023unsolved}]
\label{prob:LN-threshold}
For each integer $k\ge 3$, determine the optimal exponent $\beta=\beta(k)$ with the following property: there is a constant $c>0$, independent of $k$ and $n$, such that every graph $G$ on $n\ge c k^\beta$ vertices with
$\lambda(G)>\lambda(S^+_{n,k})$ contains a copy of $C_{2k+2}$.
\end{problem}

The problem above is stated for the adjacency spectral radius. Our result is naturally formulated in the $A_\alpha$ setting introduced earlier. For fixed $\varepsilon>0$ and an integer $k\ge3$, let $N_\varepsilon(k)$ be the least integer $N$ with the following property: for every $n\ge N$ and every $0\le \alpha\le 1-\varepsilon$, if $G$ is an $n$-vertex $C_{2k+2}$-free graph, then $ \rho_\alpha(G)\le \rho_\alpha(S^+_{n,k}),$ with equality if and only if $G\cong S^+_{n,k}$. Our main theorem states that this threshold is linear in $k$, uniformly for all $0\le \alpha\le 1-\varepsilon$.
\begin{theorem}
\label{thm:main}
For every $\varepsilon>0$ there exist constants $C_\varepsilon>0$ and $k_\varepsilon$ such that the following holds. Let $0\le \alpha\le 1-\varepsilon$, $k\ge k_\varepsilon$ and let $n\ge C_\varepsilon k$. If $G$ is an $n$-vertex $C_{2k+2}$-free graph, then $ \rho_\alpha(G)\le \rho_\alpha(S^+_{n,k})$ with equality if and only if $G\cong S^+_{n,k}$. Consequently, $ N_\varepsilon(k)=\Theta_\varepsilon(k)$.
\end{theorem}

Taking $\alpha=0$ recovers the adjacency spectral radius. Therefore, for all sufficiently large $k$, Theorem~\ref{thm:main} answers Problem~\ref{prob:LN-threshold}: the required threshold has linear order in $k$. In particular, the optimal exponent is $\beta=1$ in this range. The assumption that $\alpha$ stays away from $1$ is needed for the uniqueness statement. When $\alpha=1$, we have $A_1(G)=D(G)$, and hence $\rho_1(G)$ is the maximum degree of $G$. Both $S^+_{n,k}$ and the star $K_{1,n-1}$ have maximum degree $n-1$. However, $K_{1,n-1}$ is $C_{2k+2}$-free and is not isomorphic to $S^+_{n,k}$. Thus the uniqueness conclusion of Theorem~\ref{thm:main} cannot hold at $\alpha=1$.

We now outline the proof. The key idea is a weighted rooted version of the Erd\H{o}s--Gallai path theorem. Roughly speaking, it says that if a graph has no path of order $2s+1$ whose two endpoints lie in a prescribed root set $U$, then the total weighted incidence from $U$ is bounded by a linear expression in the relevant total weights. The lemma is useful in two ways. With Perron-vector weights, it gives estimates for $\rho_\alpha(G)$; with indicator weights, it gives the corresponding unweighted structural estimates. The precise statement and proofs are given in Section~2.2.

The same method also gives two applications. First, applying Theorem~\ref{thm:main} with smaller parameters yields a spectral condition forcing all even cycle lengths in a prescribed interval. Second, the weighted rooted estimates developed in the proof yield asymptotically tight \(A_\alpha\)-spectral bounds for \((K_1\vee P_\ell)\)-free graphs and for \(F_s\)-free graphs, where \(F_s\) denotes the friendship graph. These applications are proved in Section~4.

The rest of the paper is organized as follows. Section~2 gives the $A_\alpha$ preliminaries, the comparison graphs, the weighted rooted Erd\H{o}s--Gallai lemmas, and the switching lemmas. In Section~3, we prove Theorem~\ref{thm:main}. Section~4 gives the applications to consecutive even cycles and to local-density spectral estimates.

\section{Preliminaries}

All graphs are finite and simple. For a graph $G$, let $V(G)$ and $E(G)$ denote its vertex set and edge set, and let $|G|=|V(G)|$. The order of a graph or subgraph refers to its number of vertices. For $X\subseteq V(G)$, let $G[X]$ be the subgraph of $G$ induced by $X$, and write $G-X=G[V(G)\setminus X]$. For $u\in V(G)$, let $N_G(u)$ be the neighborhood of $u$, and let $d_G(u)=|N_G(u)|$. If $X\subseteq V(G)$ and $u\in V(G)$, let $d_X(u)=|N_G(u)\cap X|$. When the graph is clear, we omit the subscript $G$.

For $X\subseteq V(G)$, let $e(X)=e(G[X])$. If $X,Y\subseteq V(G)$ are disjoint, let $e(X,Y)$ be the number of edges with one endpoint in $X$ and the other in $Y$. For any $X,Y\subseteq V(G)$, write $I(X,Y)=\sum_{x\in X}d_Y(x)$. Then $I(X,Y)=e(X,Y)$ when $X$ and $Y$ are disjoint, and $I(X,X)=2e(X)$.

We also use some notation for weights and vectors. If $f$ is a nonnegative function on a finite set $\Omega$, let $\|f\|_\infty=\max_{\omega\in\Omega}f(\omega)$. If $v\in\mathbb R^{V(G)}$ and $X\subseteq V(G)$, then $v_X$ denotes the restriction of $v$ to $X$. We write $\mathbf 1_X$ for the all-one vector on $X$.

We use the notation $0<a_1\ll_\eps a_2\ll_\eps\cdots\ll_\eps a_m\ll_\eps 1$ in the usual way: the constants are chosen from right to left, each small enough in terms of $\eps$ and the constants to its right. After these constants are fixed, $C_\eps$ and $k_\eps$ are chosen large enough. The symbols $K_\eps,K_\eps',K_\eps''$ denote positive constants depending only on $\eps$; their values may change from line to line.

Throughout the proof, the constants are chosen in the order $0<\mu\ll_\eps \tau\ll_\eps \eta\ll_\eps \gamma\ll_\eps \xi\ll_\eps 1$, $D_\eps\gg_\eps 1$, $C_\eps\gg_\eps D_\eps$, and $k_\eps\gg_\eps 1$. We use this only in the following simple ways. After $\mu,\tau,\eta,\gamma,\xi$ are fixed, we choose $D_\eps$, $C_\eps$, and $k_\eps$ so that, whenever $1\le r\le \gamma k$, $n\ge C_\eps k$, and $k\ge k_\eps$, we have
\begin{flalign*}
    \begin{cases}
        K_\eps\mu kn \le \eta\gamma kn,\\
        K_\eps\tau n \le \eta n,\\
        K_\eps\eta kn \le \xi kn,\\
        K_\eps k\sqrt{kn} \le \xi kn,\\
        K_\eps k^2 \le \xi kn,\\
        K_\eps r\sqrt{kn}+K_\eps r^2+K_\eps r \le \dfrac{\eta rn}{4}.
\end{cases}
\end{flalign*}
We shall use the bound $W(v)\le \mu n$ only after Lemma~\ref{lem:mass-alpha} has been applied.

\subsection{\texorpdfstring{$A_\alpha$}{A-alpha}-estimates for comparison graphs}

In this subsection we collect the spectral estimates used later. Throughout, put $\beta=1-\alpha$. In particular, if $0\le \alpha\le 1-\eps$, then $\beta\ge \eps$. For an edge $uv$ and a vector $x$, define
$
        \Phi_\alpha(x_u,x_v):=\alpha(x_u^2+x_v^2)+2\beta x_ux_v.
$
Then
$
        x^TA_\alpha(G)x=\sum_{uv\in E(G)}\Phi_\alpha(x_u,x_v).
$
We shall use the following two basic facts without proofs.

\begin{lemma}
\label{lem:matrix}
Let $M$ be a nonnegative symmetric matrix, let $y$ be a nonnegative nonzero vector, and let $c>0$. If $My\ge cy$ entrywise, then $\lambda(M)\ge c$.
\end{lemma}

\begin{lemma}
\label{lem:strict-monotonicity}
Let $0\le\alpha<1$. If $G$ is a proper subgraph of a connected graph $H$ on the same vertex set, then $\rho_\alpha(G)<\rho_\alpha(H)$.
\end{lemma}

The second lemma is used twice: first to take an extremal graph to be connected, and later to exclude proper spanning subgraphs of $S^+_{n,k}$. We begin with the comparison graph $S_{n,k}$.

\begin{lemma}
\label{lem:comparison-equation}
For $1\le j\le n$, let $a_j=\alpha(n-1)+\beta(j-1)$ and $b_j=\alpha j$. Then $\rho_\alpha(S_{n,k})$ is the larger root of
$
        (t-a_k)(t-b_k)=\beta^2 k(n-k).
$
Moreover, if $\alpha<1$, then $\rho_\alpha(S^+_{n,k})>\rho_\alpha(S_{n,k})$.
\end{lemma}

\begin{proof}
Let $L$ be the clique of $S_{n,k}$, and let $R$ be its independent part. By symmetry, the Perron vector of $A_\alpha(S_{n,k})$ is constant on $L$ and on $R$. On such vectors, the quotient matrix is
$$
        Q=
        \begin{pmatrix}
        a_k & \beta(n-k)\\
        \beta k & b_k
        \end{pmatrix}.
$$
Hence, $\rho_\alpha(S_{n,k})$ is the largest eigenvalue of $Q$. Therefore it is the larger root of
$
        \det(tI-Q)=(t-a_k)(t-b_k)-\beta^2k(n-k)=0.
$
Finally, $S_{n,k}$ is a proper spanning subgraph of the connected graph $S^+_{n,k}$. Since $\alpha<1$, by Lemma~\ref{lem:strict-monotonicity}, we have $\rho_\alpha(S^+_{n,k})>\rho_\alpha(S_{n,k})$.
\end{proof}

We shall use the following consequences of the last equation.

\begin{lemma}
\label{lem:comparison-gaps}
Let $\rho_S=\rho_\alpha(S_{n,k})$. For every $\eps>0$, there are constants $c_\eps,K_\eps>0$ such that if $0\le\alpha\le1-\eps$ and $n\ge 4k$, then
$
        0\le \rho_S-a_k\le K_\eps\sqrt{kn}
$
and
$
        \rho_S-b_k\ge c_\eps\sqrt{kn}.
$
Moreover, if $a_h=a_k-\beta r$ for some $0\le r\le k$, then
$
        \rho_S-a_h\le K_\eps(k+\sqrt{kn}).
$
\end{lemma}

\begin{proof}
Let $X=\rho_S-a_k$ and $Y=\rho_S-b_k$. By Lemma~\ref{lem:comparison-equation}, we have $XY=\beta^2k(n-k)$. Since
$
        a_k-b_k=\alpha(n-k-1)+\beta(k-1)\ge0,
$
we have $0\le X\le Y$. Hence
$
        X\le \sqrt{XY}=\beta\sqrt{k(n-k)}\le\sqrt{kn}.
$
Also, since $n\ge4k$ and $\beta\ge\eps$,
$
        Y\ge \sqrt{XY}=\beta\sqrt{k(n-k)}
        \ge \frac{\sqrt3}{2}\eps\sqrt{kn}.
$
Thus the first two estimates hold with $c_\eps=\sqrt3\,\eps/2$ and $K_\eps\ge1$. Finally,
$
        \rho_S-a_h=(\rho_S-a_k)+\beta r\le K_\eps\sqrt{kn}+k,
$
and the last estimate follows after increasing $K_\eps$.
\end{proof}

We also need the following spectral bound for even-cycle-free graphs.

\begin{lemma}[Nikiforov~\citep{Nikiforov2010}]
\label{Nikiforov2010Evencycle}
Let $F$ be a $C_{2k+2}$-free graph on $m$ vertices. Then $\lambda(F)\le \sqrt{2k(m-1)}$.
\end{lemma}

\begin{lemma}
\label{lem:spectral-upper-alpha}
Let $F$ be a $C_{2k+2}$-free graph on $m$ vertices. Then
$
        \rho_\alpha(F)\le \alpha(m-1)+\beta\sqrt{2k(m-1)}.
$
Moreover, if $k\ge2$ and $m\ge4k$, then
$
        \rho_\alpha(S_{m,k})\ge \beta\sqrt{km}.
$
\end{lemma}

\begin{proof}
By Lemma~\ref{Nikiforov2010Evencycle}, we have $\lambda(F)\le \sqrt{2k(m-1)}$. For any unit vector $x$,
\begin{flalign*}
        x^TA_\alpha(F)x
        =\alpha x^TD(F)x+\beta x^TA(F)x 
        \le \alpha\Delta(F)+\beta\lambda(F) 
        \le \alpha(m-1)+\beta\sqrt{2k(m-1)}. 
\end{flalign*}
Taking the maximum over all unit vectors $x$ proves the first bound.

Now we prove the second bound. Since $A_\alpha(S_{m,k})\ge \beta A(S_{m,k})$ entrywise, by Lemma~\ref{lem:matrix}, we have
$
        \rho_\alpha(S_{m,k})\ge \beta\lambda(S_{m,k}).
$
The adjacency spectral radius of $S_{m,k}$ is the larger root of
$
        t^2-(k-1)t-k(m-k)=0.
$
Hence
$
        \lambda(S_{m,k})
        =\frac{k-1+\sqrt{(k-1)^2+4k(m-k)}}{2}.
$
We claim that $\lambda(S_{m,k})\ge\sqrt{km}$ when $k\ge2$ and $m\ge4k$. Since $2\sqrt{km}-(k-1)>0$, this is equivalent to
$
        (k-1)^2+4k(m-k)
        \ge \bigl(2\sqrt{km}-(k-1)\bigr)^2.
$
Hence we have
$
        (k-1)\sqrt{km}\ge k^2.
$
This follows from $m\ge4k$, since
$
        (k-1)\sqrt{km}\ge 2k(k-1)\ge k^2
$
for every $k\ge2$. Therefore
$
        \rho_\alpha(S_{m,k})
        \ge \beta\lambda(S_{m,k})
        \ge \beta\sqrt{km}.
$
\end{proof}

We shall also use the following edge bound.

\begin{lemma}[Verstra\"ete~\citep{Verstraete2000}]
\label{lem:even-circuit}
For every $k\ge1$, $\ex(n,C_{2k+2})\le 8k n^{1+1/(k+1)}$.
\end{lemma}

The next lemma gives a bound on the total mass of a Perron vector.

\begin{lemma}
\label{lem:mass-alpha}
For every $\eps>0$ and every $\sigma>0$, there are constants $C_{\eps,\sigma}$ and $k_{\eps,\sigma}$ such that the following holds. Let $0\le\alpha\le1-\eps$, let $k\ge k_{\eps,\sigma}$, and let $n\ge C_{\eps,\sigma}k$. Let $H$ be an $n$-vertex $C_{2k+2}$-free graph with
$
        \rho_\alpha(H)\ge \rho_\alpha(S_{n,k}).
$
If $v$ is a nonnegative Perron vector of $A_\alpha(H)$ normalized by $\max_u v_u=1$, then
$
        W(v):=\sum_{u\in V(H)}v_u\le \sigma n.
$
\end{lemma}

\begin{proof}
By Lemma~\ref{lem:spectral-upper-alpha},
$
        \rho_\alpha(H)\ge \rho_\alpha(S_{n,k})
        \ge \beta\sqrt{kn}
        \ge \eps\sqrt{kn},
$
provided $n\ge4k$. Let $v$ be a nonnegative Perron vector of $A_\alpha(H)$ normalized by $\max_u v_u=1$, and write $W(v)=\sum_{u\in V(H)}v_u$. For every vertex $u$, the Perron equation gives
\begin{flalign*}
        \rho_\alpha(H)v_u
        =\alpha d(u)v_u+\beta\sum_{w\sim u}v_w 
        \le \alpha d(u)+\beta d(u) 
        =d(u). 
\end{flalign*}
Summing over all vertices and using Lemma~\ref{lem:even-circuit}, we obtain
\begin{flalign*}
        \eps\sqrt{kn}\,W(v)
    \le \rho_\alpha(H)W(v) 
    \le \sum_{u\in V(H)}d(u) 
    =2e(H) 
    \le 16k n^{1+1/(k+1)}. 
\end{flalign*}
Therefore
$
        \frac{W(v)}{n}
        \le 16\eps^{-1}\sqrt{k}\,n^{-1/2+1/(k+1)}.
$
Choose $C_{\eps,\sigma}$ large enough so that $C_{\eps,\sigma}\ge4$ and
$
        32\eps^{-1}C_{\eps,\sigma}^{-1/3}\le \sigma.
$
Then choose $k_{\eps,\sigma}$ large enough so that, for every $k\ge k_{\eps,\sigma}$ and every $n\ge C_{\eps,\sigma}k$,
$
        \sqrt{k}\,n^{-1/2+1/(k+1)}
        \le 2 C_{\eps,\sigma}^{-1/3}.
$
It follows that $W(v)\le \sigma n$.
\end{proof}

\subsection{Weighted rooted Erd\H{o}s--Gallai lemmas}

This subsection proves several rooted path estimates. They are weighted variants of the Erd\H{o}s--Gallai argument: forbidding a path of prescribed order with both end vertices in a fixed set $U$ gives a linear bound on the incidences from $U$. We first recall the standard Erd\H{o}s--Gallai bound in the form used below.

\begin{lemma}[Erd\H{o}s--Gallai~\cite{ErdosGallai1959}]
\label{lem:EG}
If a graph on $n$ vertices contains no path of order $\ell$, then it has at most $(\ell-2)n/2$ edges.
\end{lemma}
The following rooted form, due to Nikiforov~\cite{Nikiforov2010}, shows that if the incidence from a prescribed vertex set is sufficiently large, then there is a long path with both end vertices in that set.
\begin{lemma}[Nikiforov~\cite{Nikiforov2010}]
\label{lem:nik-path}
Let the vertex set of a graph $G$ be partitioned into two sets $U$ and $W$. If
$$
2e(U)+e(U,W)>(2s-1)|U|+s|W|,
$$
then $G$ contains a path of order $2s+1$ with both end vertices in $U$.
\end{lemma}

We shall also use the following long even cycle theorem of Voss and Zuluaga~\cite{VossZuluaga1977}.

\begin{lemma}[Voss-Zuluaga~\cite{VossZuluaga1977}]\label{VossZuluaga} If $F$ is a $2$-connected graph with $\delta(F)\ge d\ge3$ and $|V(F)|\ge2d$, then $F$ contains an even cycle of length at least $2d$.
\end{lemma}
We first prove the incidence estimate that will be used in the weighted
arguments.
\begin{lemma}
\label{lem:rooted-set}
Let $s\ge1$ be an integer. Let $G$ be a graph, and let $U\subseteq V(G)$. Suppose that $G$ contains no path of order $2s+1$ with both end vertices in $U$. Then, for every $T\subseteq V(G)$,
$$
\sum_{t\in T}d_U(t)\le (s-1)|U|+s|T|.
$$
\end{lemma}

\begin{proof}
We first prove an auxiliary assertion. Let $H$ be a connected graph, let $A\subseteq V(H)$ and $B=V(H)\setminus A$. Suppose that $B$ is independent, every vertex of $A$ has degree at least $s$, and every vertex of $B$ has at least $s+1$ neighbors in $A$. If $|V(H)|\ge2s+1$, then $H$ contains a path of order $2s+1$ with both end vertices in $A$.

For $s=1$ the assertion is immediate. Thus we may assume that $s\ge2$. Add a new vertex $z$ adjacent to every vertex of $A$, and let $J$ be the resulting graph. We claim that $J$ is $2$-connected. Indeed, deleting $z$ leaves the connected graph $H$. If a vertex $a\in A$ is deleted, then $z$ is adjacent to every vertex of $A\setminus\{a\}$, and each vertex of $B$ still has a neighbor in $A\setminus\{a\}$. If a vertex of $B$ is deleted, then the remaining graph is connected because $z$ is adjacent to all vertices of $A$ and every remaining vertex of $B$ has a neighbor in $A$. Thus $J$ is $2$-connected. Moreover $\delta(J)\ge s+1$. By Lemma~\ref{VossZuluaga}, $J$ contains an even cycle $C$ of length at least $2s+2$. Suppose first that $z\notin V(C)$, then $C\subseteq H$. 
We denote $C=q_0q_1\cdots q_{\ell-1}q_0$, where $\ell\ge2s+2$ is even, and take indices modulo $\ell$. 
Let $S=\{i:q_i\in A\}$, then $|S|\ge \ell/2$, since $B$ is independent, no two consecutive vertices of $C$ lie in $B$.
If no two vertices of $A$ are joined by an arc of length $2s$ on $C$, then the shift $i\mapsto i+2s$ sends $S$ into its complement. Hence $|S|\le\ell/2$, so $|S|=\ell/2$. It follows that the vertices of $C$ alternate between $A$ and $B$. Since $2s$ is even, the shift $i\mapsto i+2s$ preserves the two parity classes of the cycle, and therefore cannot send all vertices of $S$ outside $S$. 
This contradiction shows that $C$ contains an arc of length $2s$ whose end vertices lie in $A$. This arc is the desired path. Now suppose that $z\in V(C)$. Deleting $z$ from $C$ gives a path $ P=p_0p_1\cdots p_M $ in $H$, where $p_0,p_M\in A$, $M$ is even, and $M\ge2s$. Let $S=\{i:p_i\in A\}$. Since $B$ is independent, no two consecutive indices are outside $S$. We claim that there exist $i$ and $i+2s$ both in $S$. Otherwise the shift $i\mapsto i+2s$ maps $S\cap\{0,\ldots,M-2s\}$ into $\{2s,\ldots,M\}\setminus S$.
However, since no two consecutive indices are outside $S$ and $0\in S$, we have $ |S\cap\{0,\ldots,M-2s\}|\ge \frac{M-2s}{2}+1.$
But $M\in S$, $ |\{2s,\ldots,M\}\setminus S|\le \frac{M-2s}{2}.$ 
This is impossible. Hence some $i,i+2s\in S$, and the subpath $p_ip_{i+1}\cdots p_{i+2s}$ has order $2s+1$ and both end vertices in $A$. We now prove the lemma. For a real number $x$, write $x^+=\max\{x,0\}$. It suffices to prove \[ \sum_{v\in V(G)}(d_U(v)-s)^+\le (s-1)|U|. \] Indeed, this estimate gives \[ \sum_{t\in T}d_U(t) \le s|T|+\sum_{v\in V(G)}(d_U(v)-s)^+ \le (s-1)|U|+s|T|. \] 
We prove the positive-part estimate by induction on $|U|$. Let $ X=\{v\in V(G):d_U(v)>s\}. $ If some $u\in U$ satisfies $d_X(u)\le s-1$, then we apply the induction hypothesis to $U\setminus\{u\}$. Removing $u$ from the root set decreases $(d_U(v)-s)^+$ by one precisely for vertices $v\in N(u)\cap X$, and does not increase any other term. Hence \[ \sum_{v\in V(G)}(d_U(v)-s)^+ \le (s-1)(|U|-1)+d_X(u) \le (s-1)|U|. \] We may therefore assume that every $u\in U$ has at least $s$ neighbors in $X$. Let $H$ be the graph with vertex set $U\cup X$ and edge set $ E(H)=\{xy\in E(G):x\in U,\ y\in X\}. $ Thus every vertex of $U$ has degree at least $s$ in $H$, and every vertex of $X$ has at least $s+1$ neighbors in $U$. 
By the auxiliary assertion, every component of $H$ has at most $2s$ vertices. Let $C$ be a component of $H$. Define $U_C=U\cap V(C)$, $X_C=X\cap V(C)$ and write $a=|U_C|$, $b=|X_C\setminus U_C|$, $c=|X_C\cap U_C|$. Since $|V(C)|\le2s$, we have $a+b\le2s$. For every $x\in X_C$, all neighbors of $x$ in $U$ lie in $U_C$.
Hence \[ \sum_{x\in X_C}d_U(x)\le ab+(a-1)c,\] and therefore \[ \sum_{x\in X_C}(d_U(x)-s) \le ab+(a-1)c-s(b+c). \] We claim that the right-hand side is at most $(s-1)a$. Equivalently, $ a(s-1)+b(s-a)+c(s-a+1)\ge0. $ If $a\le s$, this is immediate. If $a\ge s+1$, then the coefficient of $c$ is nonpositive, and replacing $c$ by the larger value $a$ can only decrease the left-hand side. It is therefore enough to prove $a(s-1)+b(s-a)+a(s-a+1)\ge0. $ The left-hand side equals $ a(2s-a-b)+sb, $ which is nonnegative because $a+b\le2s$. Summing over all components gives \[ \sum_{v\in V(G)}(d_U(v)-s)^+ =\sum_{x\in X}(d_U(x)-s) \le (s-1)|U|. \]
This completes the proof of Lemma~\ref{lem:rooted-set}.
\end{proof}

The next lemma is the weighted form of Lemma~\ref{lem:rooted-set}.
\begin{lemma}
\label{lem:rooted-weighted-general}
Let $G$ be a graph, let $U\subseteq V(G)$, and suppose that $G$ contains no path of order $2s+1$ with both end vertices in $U$. If $f:U\to[0,\infty)$ and $g:V(G)\to[0,\infty)$, then
$$
\sum_{u\in U}f(u)\sum_{y\sim u}g(y) \le (s-1)\|g\|_\infty\sum_{u\in U}f(u) +s\|f\|_\infty\sum_{y\in V(G)}g(y).
$$
\end{lemma}

\begin{proof}
If $\|f\|_\infty=0$ or $\|g\|_\infty=0$, then we are done. By homogeneity assume $0\le f,g\le1$. For $0\le p,q\le1$, let
$
U_p=\{u\in U:f(u)\ge p\}$ and $T_q=\{y\in V(G):g(y)\ge q\}$.

By Lemma~\ref{lem:rooted-set}, applied with $T=T_q$ and $U=U_p$,
$$
\sum_{y\in T_q}d_{U_p}(y)\le (s-1)|U_p|+s|T_q|.
$$
Integrating the preceding inequality over $(p,q)\in[0,1]^2$, and using
the identities
$
        \int_0^1 |U_p|\,dp=\sum_{u\in U}f(u)$,
        $\int_0^1 |T_q|\,dq=\sum_{y\in V(G)}g(y),
$
and
$
        \int_0^1\int_0^1\sum_{y\in T_q}d_{U_p}(y)\,dp\,dq
        =\sum_{u\in U}f(u)\sum_{y\sim u}g(y),
$
gives the desired inequality in the normalized case $0\le f,g\le1$.
The general case follows by rescaling.
\end{proof}

We shall use the following consequence of Lemma~\ref{lem:rooted-weighted-general}.

\begin{lemma}
\label{lem:two-level}
Let $G$ be a graph, let $U\subseteq V(G)$, and suppose that $G$ contains no
path of order $2s+1$ with both end vertices in $U$. Let $A,M\subseteq V(G)$
be disjoint, let $a=|A|$, and let $\omega:M\to[0,\eta]$, where
$0<\eta<1$. Then
\[
        \sum_{u\in U}d_A(u)
        +\sum_{u\in U}\sum_{m\in N(u)\cap M}\omega(m)
        \le
        \bigl(\eta(s-1)+(1-\eta)a\bigr)|U|
        +\eta sa+s\sum_{m\in M}\omega(m).
\]
\end{lemma}

\begin{proof}
Define $\widetilde\omega:V(G)\to[0,1]$ with
$\widetilde\omega(y)=1$ for $y\in A$, $\widetilde\omega(y)=\omega(y)/\eta$
for $y\in M$, and $\widetilde\omega(y)=0$ otherwise. By
Lemma~\ref{lem:rooted-weighted-general} with $f\equiv1$ on $U$ and
$g=\widetilde\omega$, we have
\[
        \sum_{u\in U}d_A(u)
        +\eta^{-1}\sum_{u\in U}\sum_{m\in N(u)\cap M}\omega(m)
        \le
        (s-1)|U|+sa+s\eta^{-1}\sum_{m\in M}\omega(m).
\]
Multiplying by $\eta$ yields
\[
        \eta\sum_{u\in U}d_A(u)
        +\sum_{u\in U}\sum_{m\in N(u)\cap M}\omega(m)
        \le
        \eta(s-1)|U|+\eta sa+s\sum_{m\in M}\omega(m).
\]
By $(1-\eta)\sum_{u\in U}d_A(u)\le(1-\eta)a|U|$, Lemma \ref{lem:two-level} holds.
\end{proof}

\subsection{\texorpdfstring{$A_\alpha$}{A-alpha}-Rayleigh switching}

This subsection introduces the $A_\alpha$ switching identities used in the
extremal comparisons.

\begin{lemma}
\label{lem:aalpha-switch}
Let $G$ be a graph, let $u\in V(G)$, and let $G'$ be obtained from $G$ by changing only the edges incident with $u$. For every vector $x$,
$$
        x^T(A_\alpha(G')-A_\alpha(G))x
        =
        \sum_{y\in N_{G'}(u)\setminus N_G(u)}\Phi_\alpha(x_u,x_y)
        -
        \sum_{y\in N_G(u)\setminus N_{G'}(u)}\Phi_\alpha(x_u,x_y).
$$
In particular, if the right side is positive at a Perron vector of $G$, then $\rho_\alpha(G')>\rho_\alpha(G)$.
\end{lemma}

\begin{proof}
The identity follows by summing the edge contribution $\Phi_\alpha(x_a,x_b)$ over the edges that are added or deleted. Let $x$ be a nonzero Perron vector of $A_\alpha(G)$. If the difference is positive at $x$, then \[ \frac{x^TA_\alpha(G')x}{x^Tx} > \frac{x^TA_\alpha(G)x}{x^Tx} = \rho_\alpha(G). \] The Rayleigh principle gives $\rho_\alpha(G')>\rho_\alpha(G)$.
\end{proof}

The following lemma applies the switching identity to an extremal $C_{2k+2}$-free graph.

\begin{lemma}
\label{lem:no-new-cycle-switch}
Let $H$ be a $C_{2k+2}$-free graph, and let $L\subseteq V(H)$ be such that any two distinct vertices of $L$ have more than $2k+2$ common neighbors in $H$. Let $u\in V(H)\setminus L$, and let $H'$ be obtained from $H$ by replacing $N_H(u)$ with $L$. Then $H'$ is $C_{2k+2}$-free.
\end{lemma}

\begin{proof}
Suppose that $H'$ contains a cycle $C$ of order $2k+2$. Since $H$ is $C_{2k+2}$-free, the cycle $C$ uses at least one edge incident with $u$ that is not present in $H$. 
Then in graph $H'$, the two neighbors of $u$ on $C$ are two distinct vertices $a,b\in L$. By assumption, $a$ and $b$ have more than $2k+2$ common neighbors in $H$, so we may choose $q\in N_H(a)\cap N_H(b)$ with $q\notin V(C)$. Replacing the segment $aub$ of $C$ by $aqb$ gives a cycle of order $2k+2$ in $H$, a contradiction.
\end{proof}

\begin{lemma} \label{lem:switch-to-L} 
Let $0\le\alpha<1$, and let $H$ be an $n$-vertex $C_{2k+2}$-free graph with maximum $A_\alpha$-spectral radius among all $n$-vertex $C_{2k+2}$-free graphs. Let $v$ be a positive Perron vector of $A_\alpha(H)$, and let $L\subseteq V(H)$. Suppose that, for some $u\in V(H)\setminus L$, the graph obtained from $H$ by replacing $N_H(u)$ with $L$ is still $C_{2k+2}$-free. Then \[ \sum_{y\in N_H(u)}\Phi_\alpha(v_u,v_y) \ge \sum_{\ell\in L}\Phi_\alpha(v_u,v_\ell). \] 
\end{lemma} 
\begin{proof}
Let $H'$ be the graph obtained from $H$ by deleting all edges from $u$ to $N_H(u)\setminus L$ and adding all missing edges from $u$ to $L$. By assumption, $H'$ is $C_{2k+2}$-free. If $\sum_{y\in N_H(u)}\Phi_\alpha(v_u,v_y) < \sum_{\ell\in L}\Phi_\alpha(v_u,v_\ell)$, then by Lemma~\ref{lem:aalpha-switch} applied at the Perron vector $v$,  we have $\rho_\alpha(H')>\rho_\alpha(H)$. This contradicts the extremality of $H$. \end{proof}

The following consequence of Lemma~\ref{lem:switch-to-L} gives a uniform
lower bound on the Perron mass of a vertex whose neighborhood cannot be
improved by switching to $L$.

\begin{lemma}
\label{lem:switching-lower-rhov}
Fix $\eps>0$. There are constants $\eta_0,c_\eps>0$ such that the
following holds for every $0\le\alpha\le1-\eps$. Let $\beta=1-\alpha$.
Let $H$ be a graph, and let $v$ be a positive Perron vector of
$A_\alpha(H)$. Let $L\subseteq V(H)$ be a set of size $h$ such that
$v_\ell\ge1-\eta_0$ for every $\ell\in L$. If $u\notin L$ satisfies
\[
        \sum_{y\sim u}\Phi_\alpha(v_u,v_y)
        \ge
        \sum_{\ell\in L}\Phi_\alpha(v_u,v_\ell),
\]
then $\rho_\alpha(H)v_u\ge c_\eps h$.
\end{lemma}

\begin{proof}
Let $s=v_u$. Since $v_\ell\ge1-\eta_0$ for every $\ell\in L$, we have
\[
        \sum_{\ell\in L}\Phi_\alpha(s,v_\ell)
        \ge h\bigl(\alpha(1-\eta_0)^2+2\beta s(1-\eta_0)\bigr).
\]
On the other hand, the Perron equation gives
$\rho_\alpha(H)s=\alpha d(u)s+\beta\sum_{y\sim u}v_y$. Since $0\le v_y\le1$,
\[
\begin{aligned}
        \sum_{y\sim u}\Phi_\alpha(s,v_y)
        =\alpha d(u)s^2+\alpha\sum_{y\sim u}v_y^2
          +2\beta s\sum_{y\sim u}v_y  
        \le 2s\rho_\alpha(H)s+\frac{\alpha}{\beta}\rho_\alpha(H)s .
\end{aligned}
\]
Indeed, the first and third terms are at most $2s\rho_\alpha(H)s$, while
$v_y^2\le v_y$ and $\beta\sum_{y\sim u}v_y\le\rho_\alpha(H)s$ give the
bound for the middle term. Hence
\[
        \rho_\alpha(H)s
        \ge
        h\frac{\alpha(1-\eta_0)^2+2\beta s(1-\eta_0)}
        {2s+\alpha/\beta}.
\]
Choose $\eta_0\le1/2$. We claim that the last fraction is bounded below
by a positive constant depending only on $\eps$. If $\alpha\ge\beta s$,
then $2s+\alpha/\beta\le3\alpha/\beta$, and the fraction is at least
$\beta(1-\eta_0)^2/3\ge\eps/12$. If $\alpha<\beta s$, then
$2s+\alpha/\beta<3s$, and the fraction is at least
$2\beta(1-\eta_0)/3\ge\eps/3$. Thus the fraction is at least
$c_\eps:=\eps/12$, and the lemma follows.
\end{proof}

\section{Proof of Theorem~\ref{thm:main}}

\begin{proof}[\bf Proof of Theorem~\ref{thm:main}]

Fix $\eps>0$, we choose constants
\[
        0<\mu\ll_\eps \tau\ll_\eps \eta\ll_\eps \gamma\ll_\eps \xi\ll_\eps 1,
\]
with $\eta\le \eta_0$, where $\eta_0$ is given by Lemma~\ref{lem:switching-lower-rhov}. We then choose $D_\eps$ sufficiently large, followed by $C_\eps$, and finally $k_\eps$. Throughout the proof of the upper bound and equality case let $0\le\alpha\le1-\eps$, $k\ge k_\eps$, and $n\ge C_\eps k$.

Let $H$ be an $n$-vertex $C_{2k+2}$-free graph maximizing
$\rho_\alpha$. Since $S^+_{n,k}$ is $C_{2k+2}$-free,
\[
        \rho:=\rho_\alpha(H)
        \ge \rho_\alpha(S^+_{n,k})
        > \rho_\alpha(S_{n,k})=:\rho_0 .
\]
We may assume that $H$ is connected. Indeed, if $H$ were disconnected,
then joining its components by a forest would not create a new cycle, and
would give a proper connected supergraph on the same vertex set. This would
increase the $A_\alpha$-spectral radius by Lemma~\ref{lem:strict-monotonicity},
a contradiction to the choice of $H$.

Let $v$ be the positive Perron vector of $A_\alpha(H)$, normalized by
$\max_{u\in V(H)}v_u=1$. Fix a vertex $x$ with $v_x=1$, and define
$
        L=\{u\in V(H):v_u\ge1-\eta\},
$
then $x\in L$.

We first prove an algebraic consequence of the comparison equation.

\begin{claim}
\label{lem:rho-lower-shift}
For every $t\ge \rho_0$, we have
$
        t^2\ge a_kt+\beta^2k(n-k)\ge \alpha(n-1)t+\beta^2k(n-k).
$
\end{claim}

\begin{proof}[ Proof of Claim~\ref{lem:rho-lower-shift}]
By Lemma~\ref{lem:comparison-equation}, $(\rho_0-a_k)(\rho_0-b_k)=\beta^2k(n-k)$. Hence \[\rho_0^2-a_k\rho_0=b_k(\rho_0-a_k)+\beta^2k(n-k)\ge \beta^2k(n-k).\]
The function $f(t)=t^2-a_kt$ is increasing on $[\rho_0,\infty)$,
since $\rho_0\ge a_k$ and $f'(t)=2t-a_k\ge 2\rho_0-a_k$.
Thus for every $t\ge\rho_0$, we have $t^2-a_kt\ge \rho_0^2-a_k\rho_0\ge \beta^2k(n-k)$, and so $t^2\ge a_kt+\beta^2k(n-k)$. Finally, $a_k=\alpha(n-1)+\beta(k-1)\ge\alpha(n-1)$, which gives $a_kt+\beta^2k(n-k)\ge\alpha(n-1)t+\beta^2k(n-k)$.
\end{proof}

Applying Lemma~\ref{lem:mass-alpha} with $\sigma=\mu$, and increasing
$C_\eps$ and $k_\eps$ if necessary, we shall use throughout that
\[
        W(v):=\sum_{u\in V(H)}v_u\le \mu n.
\]
In particular, since $\mu\le\eta$, we have $W(v)\le\eta n$.

\begin{claim}
\label{lem:heavy-degree-alpha}
Every $z\in L$ satisfies
$
d(z)\ge (1-\xi)n.
$
Thus, any two vertices of $L$ have at least $(1-2\xi)n$ common neighbors.
\end{claim}

\begin{proof}[Proof of Claim~\ref{lem:heavy-degree-alpha}]
Let $z\in L$ and $c=v_z$, then we have $c\ge1-\eta$. 
Let $U=N(z)$ and $W=N_2(z)$, where $N_2(z)$ denotes the set of vertices at distance two from $z$.
The graph $H[U\cup W]$ contains no path of order $2k+1$ with both end vertices in $U$; for otherwise adding $z$ to such a path gives a copy of $C_{2k+2}$.

Applying Lemma~\ref{lem:rooted-weighted-general} in $H[U\cup W]$ with root set $U$, parameter $k$, $f\equiv1$, and $g=v$, we get
$$
\sum_{u\in U}\sum_{\substack{y\in U\cup W\\uy\in E(H)}}v_y\le (k-1)|U|+kW(v).
$$
Since every $u\in U$ is adjacent to $z$, this implies $\sum_{u\in U}\sum_{y\sim u}v_y\le (k-1+c)d(z)+kW(v)$. Applying Lemma~\ref{lem:rooted-weighted-general} again with $f(u)=v_u$ on $U$ and $g\equiv1$ on $U\cup W$, we obtain $\sum_{u\in U}v_u d(u)\le k\sum_{u\in U}v_u+kn\le k\rho/\beta+kn$, where the last inequality follows from the Perron equation at $z$, for $\beta\sum_{u\in U}v_u\le\rho c\le\rho$.

Expanding $A_\alpha^2v=\rho^2v$ at $z$, we obtain
\[
\rho^2c=\alpha\rho c d(z)+\alpha\beta\sum_{u\in U}d(u)v_u+\beta^2\sum_{u\in U}\sum_{y\sim u}v_y\le\bigl(\alpha\rho c+\beta^2(k-1+c)\bigr)d(z)+\alpha k\rho+\alpha\beta kn+\beta^2kW(v).
\]
By Claim~\ref{lem:rho-lower-shift}, $\rho^2c\ge \alpha(n-1)\rho c+\beta^2ck(n-k)$. Let $B_z=\alpha\rho c+\beta^2(k-1+c)$. Combining the last two inequalities, we have
\[
B_z(n-d(z))\le K_\eps\eta kn+K_\eps\bigl(k\rho+kW(v)+k^2\bigr).
\]
Indeed, $c\ge1-\eta$, the terms produced by replacing $c$ with $1$ are bounded by $K_\eps\eta kn$, and $\alpha\beta kn\le 2k\rho$ for $n\ge C_\eps k$ and $\rho\ge a_k\ge\alpha(n-1)$.

If $\alpha\ge\tau$, then $B_z\ge\alpha\rho c\ge \tau^2 n/3$ for $n\ge C_\eps k$, while $W(v)\le\mu n$ and $\rho\le n$ give $K_\eps\eta kn+K_\eps(k\rho+kW(v)+k^2)\le K_\eps kn$; choosing $C_\eps$ sufficiently large such that $K_\eps kn\le \xi nB_z$. If $\alpha<\tau$, then $B_z\ge\beta^2(k-1+c)\ge \eps^2k/2$ for $k\ge k_\eps$, and Lemma~\ref{lem:spectral-upper-alpha} gives $\rho\le \alpha(n-1)+\beta\sqrt{2k(n-1)}\le \tau n+2\sqrt{kn}$. Hence the right side is at most $K_\eps(\eta kn+\tau kn+\mu kn+k\sqrt{kn}+k^2)$, which is at most $\xi nB_z$ by the hierarchy and by the choice of $C_\eps$.
Therefore $n-d(z)\le\xi n$ and we have  $d(z)\ge(1-\xi)n$. The common neighbor bound follows from $|N(a)\cap N(b)|\ge d(a)+d(b)-n$.
\end{proof}

\begin{claim}
\label{lem:L-upper}
$|L|\le k$.
\end{claim}

\begin{proof}[Proof of Claim~\ref{lem:L-upper}]
If $|L|\ge k+1$, choose distinct vertices $h_1, \ldots,h_{k+1}\in L$. By Claim~\ref{lem:heavy-degree-alpha}, each consecutive pair $h_i,h_{i+1}$, with indices modulo $k+1$, has at least $(1-2\xi)n$ common neighbors. Since $n\ge C_\eps k$ and $\xi$ is small, we may greedily choose distinct vertices
$
r_i\in N(h_i)\cap N(h_{i+1}),
$
$i=1,\ldots,k+1$,
avoiding all previously chosen vertices and all $h_j$. Then
$
h_1r_1h_2r_2\cdots h_{k+1}r_{k+1}h_1
$
is a copy of $C_{2k+2}$, a contradiction.
\end{proof}

Now we let
$
h:=|L|$, and $ r:=k-h.
$
By Claim~\ref{lem:L-upper}, $r\ge0$.

\begin{claim}
\label{lem:first-defect}
We have
$
        r\le \gamma k.
$
Moreover, if $\alpha<\tau$ and $n\ge D_\eps k^3$, then $r=0$.
\end{claim}

\begin{proof}[Proof of Claim~\ref{lem:first-defect}]
Assume $r>0$. Let $x\in L$ satisfy $v_x=1$, and let $X=N(x)$, $A=L\setminus\{x\}$, and $M=V(H)\setminus L$. Write $D_X=|X|$, $S_X=\sum_{u\in X}v_u$, and $W_M=\sum_{m\in M}v_m$. Every vertex of $M$ has Perron coordinate less than $1-\eta$. The graph $H-x$ contains no path of order $2k+1$ with both endpoints in $X$, for otherwise adding $x$ gives a copy of $C_{2k+2}$.

Taking $A=L\setminus\{x\}$, $M=V(H)\setminus L$, $\omega(m)=v_m$, 
and applying Lemma~\ref{lem:two-level} with parameter $1-\eta$,
we obtain
\begin{equation*}
\label{eq:first-defect-two-level}
\sum_{u\in X}\sum_{\substack{y\in A\cup M\\uy\in E(H)}}v_y\le (k-1-\eta r)D_X+(1-\eta)k(h-1)+kW_M.
\end{equation*}
Since each $u\in X$ is adjacent to $x$, it follows that
\begin{equation*}
\label{eq:first-defect-neighbor-sum}
\sum_{u\in X}\sum_{y\sim u}v_y\le (k-\eta r)D_X+(1-\eta)k(h-1)+kW_M.
\end{equation*}
By Lemma~\ref{lem:rooted-weighted-general} (applied in $H-x$ with root set $X$, $f(u)=v_u$, and $g\equiv1$), we have
\begin{equation*}
\label{eq:first-defect-rooted-weighted}
\sum_{u\in X}v_u(d(u)-1)\le (k-1)S_X+k(n-1).
\end{equation*}
Hence
\begin{equation*}
\label{eq:first-defect-degree-sum}
\sum_{u\in X}d(u)v_u\le kS_X+k(n-1).
\end{equation*}
The Perron equation at $x$ is
\begin{equation}
\label{eq:first-defect-perron-x}
\rho=\alpha D_X+\beta S_X.
\end{equation}
Expanding $A_\alpha^2v=\rho^2v$ at $x$, 
and using the preceding bounds together with the Perron equation at $x$, we have
\begin{equation}
\label{eq:first-defect-rho-upper}
\rho^2\le \alpha\rho D_X+\alpha k(\rho-\alpha D_X)+\alpha\beta k(n-1)+\beta^2\bigl((k-\eta r)D_X+(1-\eta)k(h-1)+kW_M\bigr).
\end{equation}
By Claim~\ref{lem:rho-lower-shift}, we have
\begin{equation}
\label{eq:first-defect-rho-lower}
\rho^2\ge a_k\rho+\beta^2k(n-k).
\end{equation}
Subtracting \eqref{eq:first-defect-rho-lower} from \eqref{eq:first-defect-rho-upper}, we get
\begin{equation}
\label{eq:first-defect-subtracted}
0\le C_DD_X+(\alpha k-a_k)\rho+\alpha\beta k(n-1)+\beta^2\bigl((1-\eta)k(h-1)+kW_M-k(n-k)\bigr),
\end{equation}
where $C_D=\alpha\rho-\alpha^2k+\beta^2(k-\eta r)$. Since $\rho\ge\rho_0\ge a_k$, we have
\[
C_D\ge \alpha a_k-\alpha^2k+\beta^2(k-\eta r)=\alpha^2(n-k-1)+\alpha\beta(k-1)+\beta^2(k-\eta r)\ge0.
\]
Thus $D_X\le n-1$ may be used in \eqref{eq:first-defect-subtracted}. Hence
\begin{equation}
\label{eq:first-defect-after-D}
0\le C_\rho\rho+E,
\end{equation}
where $C_\rho=\alpha k-\beta(k-1)$ and
$
E=-\alpha^2k(n-1)+\beta^2(k-\eta r)(n-1)+\alpha\beta k(n-1)+\beta^2\bigl((1-\eta)k(h-1)+kW_M-k(n-k)\bigr).
$
Moreover, by \eqref{eq:first-defect-perron-x} and $v_y\le1$, we have
\begin{equation}
\label{eq:first-defect-rho-ah}
\rho\le \alpha(n-1)+\beta(h-1+W_M)=a_h+\beta W_M.
\end{equation}
If $C_\rho\ge0$, then using \eqref{eq:first-defect-rho-ah} in \eqref{eq:first-defect-after-D} and expanding with $h=k-r$ gives
\(
0\le -\beta^2\eta r(n-1)+\alpha\beta(n-1)+\beta(\alpha k+\beta)W_M+K_\eps k^2.
\)
If $C_\rho<0$, then using $\rho\ge a_k$ in \eqref{eq:first-defect-after-D} and expanding with $h=k-r$ gives
\(
0\le -\beta^2\eta r(n-1)+\alpha\beta(n-1)+\beta^2kW_M+K_\eps k^2.
\)
In both cases, since $\beta\ge\eps$ and $n\ge C_\eps k$, we obtain
\begin{equation}
\label{eq:first-defect-main}
c_\eps\eta rn\le K_\eps k^2+K_\eps kW_M+K_\eps\alpha n.
\end{equation}

We first prove $r\le\gamma k$. Since $W_M\le W(v)\le\mu n$, \eqref{eq:first-defect-main} gives $c_\eps\eta rn\le K_\eps k^2+K_\eps\mu kn+K_\eps n$. If $r\ge\gamma k$, then the left-hand side is at least $c_\eps\eta\gamma kn$, while the right-hand side is smaller than this for $k\ge k_\eps$, $n\ge C_\eps k$, and $\mu\ll_\eps\eta\gamma$. Hence $r<\gamma k$.

Assume that $\alpha<\tau$, $n\ge D_\eps k^3$, and $r\ge1$. By the proof of Lemma~\ref{lem:mass-alpha}, we have
$
W(v)\le 16\eps^{-1}\sqrt{k}\,n^{1/2+1/(k+1)}.
$
Since $-1/2+1/(k+1)<0$ and $n\ge D_\eps k^3$, choosing $D_\eps$ and $k_\eps$ sufficiently large gives
$kW_M\le kW(v)\le \frac{c_\eps\eta}{4}n$.
Also, $K_\eps\alpha n\le K_\eps\tau n\le c_\eps\eta n/4$, and $K_\eps k^2\le c_\eps\eta n/4$, by the choices of $\tau$ and $D_\eps$. These estimates contradict \eqref{eq:first-defect-main}, since $r\ge1$. Therefore $r=0$ whenever $\alpha<\tau$ and $n\ge D_\eps k^3$.
\end{proof}

Suppose for a contradiction that
$
        h=|L|=k-r$ and $1\le r\le\gamma k.
$
Let
$
        R=V(H)\setminus L.
$
For $u\in R$, write $A(u)=N(u)\cap L$, and define
$
        F=\{u\in R:|A(u)|\ge2\}$, $Z=R\setminus F.
$

The estimate
\begin{equation} \label{WR}
   W_R:=\sum_{u\in R}v_u\le K_\eps\rho 
\end{equation}  
will be used only in Claim~\ref{lem:low-attachment-control}. If
$\alpha\ge\tau$, then $\rho\ge a_k\ge \alpha(n-1)\ge \tau n/2$ and
$W_R\le W(v)\le\mu n$, so \eqref{WR} follows. If $\alpha<\tau$, then Claim~\ref{lem:first-defect} already gives $r=0$ whenever $n\ge D_\eps k^3$;
under the present contradiction we may therefore assume $n<D_\eps k^3$.
By the Perron equation and Lemma~\ref{lem:even-circuit}, we have
$
        \rho W(v)\le 2e(H)\le16k n^{1+1/(k+1)}\le32kn,
$
for $k\ge k_\eps$. Since $\rho\ge\beta\sqrt{k(n-k)}\ge c_\eps\sqrt{kn}$, we
have $kn\le K_\eps\rho^2$, and hence $W_R\le W(v)\le K_\eps\rho$. Thus
\eqref{WR} holds in all cases relevant to this section.

\begin{claim}
\label{lem:F-path-obstruction}
The graph $H[R]$ contains no path of order $2r+3$ with both end vertices in $F$.
\end{claim}

\begin{proof}[Proof of Claim~\ref{lem:F-path-obstruction}]
Suppose there is such a path
$
P=p_0p_1\cdots p_{2r+2}$, $p_0,p_{2r+2}\in F.
$
Choose $a\in A(p_0)$ and $b\in A(p_{2r+2})$ with $a\ne b$. This is possible because both end vertices lie in $F$. List the vertices of $L$ as
$
b=\ell_1,\ell_2,\ldots,\ell_h=a.
$
For each $i=1,\ldots,h-1$, Claim~\ref{lem:heavy-degree-alpha} gives at least $(1-2\xi)n$ common neighbors of $\ell_i$ and $\ell_{i+1}$. Greedily choose pairwise distinct vertices
$
q_i\in N(\ell_i)\cap N(\ell_{i+1})
$
avoiding $P\cup L$ and all previously chosen $q_j$. Then
$
p_0p_1\cdots p_{2r+2}bq_1\ell_2q_2\cdots q_{h-1}ap_0
$
is a cycle of order
$
(2r+3)+h+(h-1)=2r+2h+2=2k+2,
$
a contradiction.
\end{proof}

The next estimate is a key application of $A_\alpha$-Rayleigh switching.

\begin{claim}
\label{lem:low-attachment-control}
There is a constant $K_\eps>0$ such that
$
        |Z|\le 3\xi n$, $Q_Z:=\sum_{z\in Z}v_z^2\le K_\eps(r+1)$
and
$$
        J_{ZZ}:=\sum_{z\in Z}\sum_{\substack{y\in Z\\zy\in E(H)}}v_y\le \xi k|Z|.
$$
\end{claim}

\begin{proof}[Proof of Claim~\ref{lem:low-attachment-control}]
The almost-dominating property gives
$
        \sum_{u\in R}(h-|A(u)|)=\sum_{\ell\in L}|R\setminus N(\ell)|\le h\xi n.
$
Every $z\in Z$ misses at least $h-1$ vertices of $L$. Since $h=k-r\ge(1-\gamma)k$ and $\gamma\ll1$, this gives $|Z|\le3\xi n$.

Let $z\in Z$. By Claim~\ref{lem:heavy-degree-alpha}, any two vertices of $L$ have more than $2k+2$ common neighbors in $H$, once $C_\eps$ is sufficiently large. Hence Lemma~\ref{lem:no-new-cycle-switch} shows that replacing $N_H(z)$ with $L$ creates no $C_{2k+2}$. By Lemmas~\ref{lem:switch-to-L} and~\ref{lem:switching-lower-rhov}, we have
\(
\rho v_z\ge c_\eps h\ge c_\eps k.
\)
Since $z$ has at most one neighbor in $L$,
$
        \alpha d_L(z)v_z+\beta\sum_{\ell\in A(z)}v_\ell\le1.
$
For $k\ge k_\eps$ this gives
$$
        \alpha d_R(z)v_z+\beta\sum_{\substack{y\in R\\zy\in E(H)}}v_y
        \ge (1-\xi)\rho v_z.
$$
Multiplying by $v_z$ and summing over $z\in Z$, we obtain
$
        (1-\xi)\rho Q_Z\le E_{ZZ}^{(\alpha)}+E_{ZF}^{(\alpha)},
$
where
$$
        E_{ZZ}^{(\alpha)}
        =
        \alpha\sum_{z\in Z}d_Z(z)v_z^2
        +
        \beta\sum_{z\in Z}v_z\sum_{\substack{y\in Z\\zy\in E(H)}}v_y
$$
and
$$
        E_{ZF}^{(\alpha)}
        =
        \alpha\sum_{z\in Z}d_F(z)v_z^2
        +
        \beta\sum_{z\in Z}v_z\sum_{\substack{f\in F\\zf\in E(H)}}v_f.
$$
By Lemma~\ref{lem:spectral-upper-alpha} and $|Z|\le3\xi n$, taking $\xi\ll_\eps1$ gives
$
        E_{ZZ}^{(\alpha)}\le \rho_\alpha(H[Z])Q_Z\le (1-3\xi)\rho Q_Z.
$
Therefore
$
        \xi\rho Q_Z\le E_{ZF}^{(\alpha)}.
$
We now bound $E_{ZF}^{(\alpha)}$. The rooted obstruction Claim~\ref{lem:F-path-obstruction} and Lemma~\ref{lem:rooted-weighted-general} applied in $H[R]$ with root set $F$ and parameter $r+1$ give
$$
        \sum_{z\in Z}v_z\sum_{\substack{f\in F\\zf\in E(H)}}v_f
        \le rW_F+(r+1)W_Z\le (r+1)W_R.
$$
Similarly, applying Lemma~\ref{lem:rooted-weighted-general} with $f\equiv1$ on $F$ and $g(z)=v_z^2\mathbf 1_Z(z)$ gives
$$
        \sum_{z\in Z}d_F(z)v_z^2
        =
        \sum_{f\in F}\sum_{\substack{z\in Z\\fz\in E(H)}}v_z^2
        \le r|F|+(r+1)Q_Z.
$$
Since $\rho\ge a_k\ge\alpha(n-1)$, we have $\alpha|F|\le K_\eps\rho$. Also, because $r\le\gamma k$ and $n\ge C_\eps k$, the term $\alpha(r+1)Q_Z$ is absorbed into the left-hand side. Using $W_R\le K_\eps\rho$, we get
$
        Q_Z\le K_\eps(r+1).
$

Finally, by Cauchy's inequality and Lemma~\ref{lem:spectral-upper-alpha},
$$
        J_{ZZ}
        =
        \mathbf 1_Z^TA(H[Z])v_Z
        \le |Z|^{1/2}\lambda(H[Z])Q_Z^{1/2}
        \le |Z|\sqrt{2kQ_Z}
        \le |Z|\sqrt{2K_\eps k(r+1)}.
$$
Since $r\le\gamma k$ and $\gamma\ll_\eps\xi^2$, the last expression is at most $\xi k|Z|$.
\end{proof}

\begin{claim}
\label{lem:aalpha-completion-ineq}
$
        (\rho-a_h)
        \bigl(\rho-\alpha h-2r-1\bigr)
        \le
        \beta\bigl(\beta h+(1-\eta)r\bigr)(n-h).
$
\end{claim}

\begin{proof}[Proof of Claim~\ref{lem:aalpha-completion-ineq}]
Let $W_R=\sum_{u\in R}v_u$ and $m=|R|=n-h$. Summing the Perron equations over all vertices of $R$, we get
\[
\rho W_R=\alpha\sum_{u\in R}d(u)v_u+\beta\sum_{u\in R}\sum_{y\sim u}v_y.
\]
We split the right-hand side into the contribution of $R$--$L$ edges and the contribution of internal $R$-edges. Since every vertex of $F$ has at most $h$ neighbors in $L$, while every vertex of $Z$ has at most one neighbor in $L$, the $R$--$L$ contribution is at most
\begin{equation}
\label{eq:completion-RL}
\alpha\sum_{u\in R}d_L(u)v_u+\beta\sum_{u\in R}\sum_{\ell\in A(u)}v_\ell\le \alpha hW_R+\beta h|F|+\beta |Z|=\alpha hW_R+\beta hm-\beta(h-1)|Z|.
\end{equation}
For an internal edge $uv\in E(H[R])$, its total contribution to the sum of the Perron equations over $R$ is $\alpha(v_u+v_v)+\beta(v_u+v_v)=v_u+v_v$. Thus the internal $R$-contribution equals $\sum_{u\in R}\sum_{\substack{y\in R\\uy\in E(H)}}v_y$. Write this double sum as $I_F+I_Z$, where $I_F=\sum_{f\in F}\sum_{\substack{y\in R\\fy\in E(H)}}v_y$ and $I_Z=\sum_{z\in Z}\sum_{\substack{y\in R\\zy\in E(H)}}v_y$.

By Claim~\ref{lem:F-path-obstruction}, $H[R]$ contains no path of order $2r+3$ with both endpoints in $F$. Applying Lemma~\ref{lem:rooted-weighted-general} in $H[R]$ with root set $F$, parameter $r+1$, $f\equiv1$, and $g=v_R$, we get
\begin{equation}
\label{eq:completion-IF}
I_F\le r\|v_R\|_\infty |F|+(r+1)W_R\le r(1-\eta)|F|+(r+1)W_R.
\end{equation}
For $I_Z$, split according to whether the neighbor lies in $F$ or $Z$. Applying Lemma~\ref{lem:rooted-weighted-general} with root set $F$, parameter $r+1$, $f=v_F$, and $g=\mathbf 1_Z$, we obtain
\begin{equation}
\label{eq:completion-FZ}
\sum_{f\in F}v_f d_Z(f)\le rW_F+(r+1)\|v_F\|_\infty |Z|\le rW_F+(r+1)(1-\eta)|Z|.
\end{equation}
The $Z$--$Z$ part is $J_{ZZ}$, and Claim~\ref{lem:low-attachment-control} gives $J_{ZZ}\le\xi k|Z|$. Combining \eqref{eq:completion-IF} and \eqref{eq:completion-FZ}, and using $W_F\le W_R$, gives
\begin{equation}
\label{eq:completion-internal}
I_F+I_Z\le (2r+1)W_R+r(1-\eta)m+(1-\eta+\xi k)|Z|.
\end{equation}
Putting \eqref{eq:completion-RL} and \eqref{eq:completion-internal} together yields
\[
\rho W_R\le (\alpha h+2r+1)W_R+\bigl(\beta h+(1-\eta)r\bigr)m+\bigl[-\beta(h-1)+1-\eta+\xi k\bigr]|Z|.
\]
Since $h=k-r$, $r\le\gamma k$, $\gamma\ll_\eps1$, $\xi\ll_\eps1$, and $\beta\ge\eps$, the coefficient of $|Z|$ is negative for $k\ge k_\eps$. Therefore
\begin{equation}
\label{eq:completion-WR-ineq}
\bigl(\rho-\alpha h-2r-1\bigr)W_R\le \bigl(\beta h+(1-\eta)r\bigr)(n-h).
\end{equation}
The factor multiplying $W_R$ is positive. Indeed, by Lemma~\ref{lem:comparison-gaps},
\[
\rho-\alpha h-2r-1\ge \rho_0-\alpha h-2r-1=(\rho_0-b_k)+\alpha r-2r-1\ge c_\eps\sqrt{kn}-3r-1>0,
\]
after choosing $C_\eps$ sufficiently large and $\gamma$ sufficiently small. Finally, the Perron equation at $x\in L$ with $v_x=1$ gives $\rho=\alpha d(x)+\beta\sum_{y\sim x}v_y\le \alpha(n-1)+\beta(h-1+W_R)=a_h+\beta W_R$. Thus $W_R\ge(\rho-a_h)/\beta$. Substituting this into \eqref{eq:completion-WR-ineq} gives
\[
(\rho-a_h)\bigl(\rho-\alpha h-2r-1\bigr)\le \beta\bigl(\beta h+(1-\eta)r\bigr)(n-h),
\]
as required.
\end{proof}

\begin{claim}
\label{lem:r-zero}
 $r=0$.
\end{claim}

\begin{proof}[Proof of Claim~\ref{lem:r-zero}]
Assume $1\le r\le\gamma k$. Put $B_h=\alpha h+2r+1$, and define
\[
P(t)=(t-a_h)(t-B_h)-\beta\bigl(\beta h+(1-\eta)r\bigr)(n-h).
\]
By Claim~\ref{lem:aalpha-completion-ineq}, $P(\rho)\le0$. We prove that $P(\rho_0)>0$ and that $P$ is increasing on $[\rho_0,\infty)$.

Let $X_0=\rho_0-a_k$ and $Y_0=\rho_0-b_k$. Then $X_0Y_0=\beta^2k(n-k)$ by Lemma~\ref{lem:comparison-equation}. Since $a_h=a_k-\beta r$, we have $\rho_0-a_h=X_0+\beta r$. Also $B_h=\alpha h+2r+1=b_k+(2-\alpha)r+1$, and hence $\rho_0-B_h=Y_0-(2-\alpha)r-1$. Using $h=k-r$ and $n-h=n-k+r$, we get
\[
P(\rho_0)=(X_0+\beta r)\bigl(Y_0-(2-\alpha)r-1\bigr)-\beta\bigl(\beta k+(\alpha-\eta)r\bigr)(n-k+r).
\]
Expanding and using $X_0Y_0=\beta^2k(n-k)$, we obtain
\[
P(\rho_0)=\beta rY_0-\bigl((2-\alpha)r+1\bigr)X_0-\beta r\bigl((2-\alpha)r+1\bigr)-\beta^2kr-\alpha\beta r(n-k+r)+\beta\eta r(n-k+r).
\]
Now $Y_0=\rho_0-b_k=\alpha(n-k-1)+\beta(k-1)+X_0$. Thus,
\[
P(\rho_0)=\beta\eta r(n-k+r)-(r+1)X_0-\beta r\bigl((2-\alpha)r+1\bigr)-\alpha\beta r(r+1)-\beta^2r.
\]
By Lemma~\ref{lem:comparison-gaps}, $X_0\le K_\eps\sqrt{kn}$. Therefore
\[
P(\rho_0)\ge \eps\eta r(n-k)-K_\eps r\sqrt{kn}-K_\eps r^2-K_\eps r.
\]
Since $r\le\gamma k$ and $n\ge C_\eps k$, the right-hand side is positive after choosing $C_\eps$ sufficiently large. Hence $P(\rho_0)>0$.

Finally, $P'(t)=2t-a_h-B_h$. For $t\ge\rho_0$, we have
\(
P'(t)\ge 2\rho_0-(a_k-\beta r)-\bigl(b_k+(2-\alpha)r+1\bigr)=(\rho_0-a_k)+(\rho_0-b_k)-r-1\ge c_\eps\sqrt{kn}-r-1>0,
\)
after choosing $C_\eps$ sufficiently large and $\gamma$ sufficiently small. Thus $P$ is increasing on $[\rho_0,\infty)$. Since $\rho>\rho_0$, we get $P(\rho)>P(\rho_0)>0$, contradicting $P(\rho)\le0$. Therefore $r=0$.
\end{proof}

Combining Claims~\ref{lem:first-defect} and~\ref{lem:r-zero}, we have
$
        |L|=k.
$
Let
$
U_0=V(H)\setminus L.
$
For $u\in U_0$, write $A(u)=N(u)\cap L$, and define
$
F=\{u\in U_0:|A(u)|\ge2\}$, and $Z=\{u\in U_0:|A(u)|\le1\}.
$

\begin{claim}
\label{lem:no-two-F-alpha}
Every vertex of $U_0$ has at most one neighbor in $F$.
\end{claim}

\begin{proof}[Proof of Claim~\ref{lem:no-two-F-alpha}]
Suppose that some $v\in U_0$ has two distinct neighbors $p,q\in F$. 
Choose $b\in A(q)$ and $a\in A(p)\setminus\{b\}$. 
We list the vertices of $L$ as $b=h_1,h_2,\ldots,h_k=a$. For each $i=1, \ldots,k-1$, the pair $h_i,h_{i+1}$ has at least $(1-2\xi)n$ common neighbors. 
Then we greedily choose distinct vertices $r_i\in N(h_i)\cap N(h_{i+1})$ avoiding $p,q,v$ and $L$. 
However the cycle
$
p\,v\,q\,h_1\,r_1\,h_2\,r_2\cdots r_{k-1}\,h_k\,p,
$
where $q h_1$ and $h_k p$ are edges by the choices of $b=h_1\in A(q)$ and $a=h_k\in A(p)$, is a copy of $C_{2k+2}$, a contradiction.
\end{proof}

\begin{claim}
\label{lem:Z-empty-alpha}
$Z=\emptyset$.
\end{claim}

\begin{proof}[Proof of Claim~\ref{lem:Z-empty-alpha}]
By Claim~\ref{lem:heavy-degree-alpha} and $|L|=k$, we have
\[
\sum_{u\in U_0}(k-|A(u)|)=\sum_{\ell\in L}|U_0\setminus N(\ell)|\le k\xi n.
\]
Since every $u\in Z$ has $k-|A(u)|\ge k-1$, it follows that $|Z|\le2\xi n$.

Assume that $Z\ne\emptyset$ and let $z\in Z$. 
By Claim~\ref{lem:heavy-degree-alpha}, for sufficiently large $C_\eps$, any two vertices of $L$ have more than $2k+2$ common neighbors in $H$.
Hence Lemma~\ref{lem:no-new-cycle-switch} shows that replacing $N_H(z)$ with $L$ creates no $C_{2k+2}$. By Lemmas~\ref{lem:switch-to-L} and~\ref{lem:switching-lower-rhov}, we have $\rho v_z\ge c_\eps k$.

The vertex $z$ has at most one neighbor in $L$ and, by Claim~\ref{lem:no-two-F-alpha}, at most one neighbor in $F$. Therefore the contribution to the Perron equation at $z$ from $L\cup F$ is at most
\[
\alpha d_{L\cup F}(z)v_z+\beta\sum_{y\in N(z)\cap(L\cup F)}v_y\le 2\alpha+2\beta=2.
\]
For $k\ge k_\eps$, this is at most $\xi\rho v_z$. Thus $A_\alpha(H[Z])v_Z\ge(1-\xi)\rho v_Z$ entrywise. 
By Lemma~\ref{lem:matrix}, we have $\rho_\alpha(H[Z])\ge(1-\xi)\rho$.

On the other hand, by Lemma~\ref{lem:spectral-upper-alpha} and $|Z|\le2\xi n$, we have
\[
\rho_\alpha(H[Z])\le \alpha|Z|+\beta\sqrt{2k|Z|}\le 2\xi\alpha n+2\beta\sqrt{\xi kn}.
\]
Since $\rho\ge a_k\ge\alpha(n-1)$ and $\rho\ge \beta\sqrt{k(n-k)}$, choosing $\xi\ll_\eps1$, $C_\eps$ and $k_\eps$ sufficiently large gives $\rho_\alpha(H[Z])<(1-2\xi)\rho$. This contradicts $\rho_\alpha(H[Z])\ge(1-\xi)\rho$. Therefore $Z=\emptyset$.
\end{proof}

\begin{claim}
\label{lem:complete-hubs-alpha}
Every vertex of $U_0$ is adjacent to every vertex of $L$, and $\Delta(H[U_0])\le1$.
\end{claim}

\begin{proof}[Proof of Claim~\ref{lem:complete-hubs-alpha}]
Since $Z=\emptyset$, $U_0\subseteq F$.
By Claim~\ref{lem:no-two-F-alpha}, every vertex of $U_0$ has at most one neighbor in $U_0$, and hence $\Delta(H[U_0])\le1$.

Suppose that some $u\in U_0$ is not adjacent to a hub in $L$. 
By Claim~\ref{lem:heavy-degree-alpha}, for sufficiently large $C_\eps$, any two vertices of $L$ have more than $2k+2$ common neighbors in $H$.
Hence Lemma~\ref{lem:no-new-cycle-switch} shows that replacing $N_H(u)$ with $L$ creates no $C_{2k+2}$. Let $H'$ be the graph obtained by this replacement.

We compare the $\Phi_\alpha$-sums at $u$. The old neighborhood of $u$ consists of $A(u)\subsetneq L$ and at most one vertex of $U_0$ and the new neighborhood is $L$. Since $u\in U_0$, we have $v_u>0$. Since $\beta>0$, the function $t\mapsto\Phi_\alpha(v_u,t)$ is strictly increasing on $[0,1]$.
Every vertex of $L$ has Perron coordinate at least $1-\eta$, whereas every vertex of $U_0$ has Perron coordinate strictly smaller than $1-\eta$. Therefore
\[
\sum_{\ell\in L}\Phi_\alpha(v_u,v_\ell)>\sum_{y\in N_H(u)}\Phi_\alpha(v_u,v_y).
\]
By Lemma~\ref{lem:aalpha-switch}, we have $\rho_\alpha(H')>\rho_\alpha(H)$, contradicting the extremality of $H$. Hence every vertex of $U_0$ is adjacent to every vertex of $L$.
\end{proof}

By Claim~\ref{lem:complete-hubs-alpha}, $H$ contains the complete bipartite graph between $L$ and $U_0$, and $H[U_0]$ has maximum degree at most $1$. If $H[U_0]$ contained two edges, then they would be disjoint, say $ab$ and $cd$. Choose distinct vertices $r_1, \ldots,r_{k-2}\in U_0\setminus\{a,b,c,d\}$ and write $L=\{\ell_1,\ldots,\ell_k\}$. Since every vertex of $U_0$ is joined to every vertex of $L$,
$
ab\ell_1r_1\ell_2r_2\cdots \ell_{k-2}r_{k-2}\ell_{k-1}cd\ell_ka
$
is a copy of $C_{2k+2}$, a contradiction. Hence
$
e(H[U_0])\le1
$
and $H$ is a spanning subgraph of a graph isomorphic to $S^+_{n,k}$. 
Since $S^+_{n,k}$ is connected and $\alpha<1$, Lemma~\ref{lem:strict-monotonicity} shows that every proper spanning subgraph has smaller $A_\alpha$-spectral radius. But
$
\rho_\alpha(H)\ge \rho_\alpha(S^+_{n,k}),
$
and thus
$
H\cong S^+_{n,k}.
$
This proves the upper bound and the equality case in Theorem~\ref{thm:main}.

We now prove the linear lower bound in Theorem~\ref{thm:main}.
For $n=2k+1$, the complete graph $K_n$ contains no $C_{2k+2}$. Moreover $K_n$ is connected and properly contains $S^+_{n,k}$ as a spanning subgraph. For every $\alpha<1$, by Lemma~\ref{lem:strict-monotonicity}, we have
$$
        \rho_\alpha(S^+_{n,k})<\rho_\alpha(K_n)=n-1.
$$
Thus the threshold cannot be $o(k)$, and in particular $N_\eps(k)\ge2k+2$. Together with the upper bound already proved, we have $N_\eps(k)=\Theta_\eps(k)$ and this completes the proof of Theorem~\ref{thm:main}.
\end{proof}

\section{Applications}
In this section we give applications of Theorem~\ref{thm:main} and of the tools developed in the proof. We first derive a consecutive-cycle consequence by applying the main theorem with all smaller parameters. We then apply the weighted rooted estimates to local-density problems and show that hereditary sparsity inside neighborhoods yields $A_\alpha$-spectral upper estimates for several related forbidden configurations.

\subsection{Consecutive even cycles}

The first application concerns intervals of cycle lengths. Woodall's theorem gives an edge condition forcing cycles of all lengths in a long interval, and Li and Ning initiated spectral analogues of this problem, proving that a spectral condition can force cycles of consecutive lengths in a linear interval \cite{Woodall1976,LiNingConsecutive2023}. Their later stability theorem and the subsequent work of Zhang further developed this direction by connecting spectral radius, maximum average degree and consecutive cycle lengths \cite{LiNing2023,Zhang2024ConsecutiveCycles,Zhang2025Skeletons}. Our consequence has a different form, since it does not start from a global threshold such as $\lambda(G)>\sqrt{\lfloor n^2/4\rfloor}$, but instead starts from the exact even-cycle threshold $\rho_\alpha(S^+_{n,k})$. Once this threshold is exceeded, Theorem~\ref{thm:main} applies with every smaller parameter, and consequently all even cycle lengths in the corresponding interval must occur.

\begin{corollary}
\label{cor:aalpha-consecutive-even-cycles}
Fix $\varepsilon>0$. Let $C_\varepsilon$ and $k_\varepsilon$ be as in Theorem~\ref{thm:main}. Let $0\le \alpha\le 1-\varepsilon$, let $k\ge k_\varepsilon$, and let $n\ge C_\varepsilon k$. If an $n$-vertex graph $G$ satisfies
$
        \rho_\alpha(G)>\rho_\alpha(S^+_{n,k}),
$
then $G$ contains $C_{2t+2}$ for every integer $t$ with $k_\varepsilon\le t\le k$. Equivalently, $G$ contains cycles of all even lengths
\[
        2k_\varepsilon+2,\ 2k_\varepsilon+4,\ \ldots,\ 2k+2.
\]
\end{corollary}

\begin{proof}
Fix an integer $t$ with $k_\varepsilon\le t\le k$. Since $n\ge C_\varepsilon k\ge C_\varepsilon t$, Theorem~\ref{thm:main} applies with parameter $t$.
If $t=k$ and $G$ contained no $C_{2k+2}$, then by Theorem~\ref{thm:main}, we have
$
        \rho_\alpha(G)\le \rho_\alpha(S^+_{n,k}),
$
a contradiction.
Thus we may assume $t<k$. We view $S^+_{n,t}$ as a proper spanning subgraph of $S^+_{n,k}$ as follows. Choose the $K_t$ part of $S^+_{n,t}$ inside the $K_k$ part of $S^+_{n,k}$, and choose the added independent-side edge to be the same added edge in the independent part of $S^+_{n,k}$. Then every edge of $S^+_{n,t}$ is an edge of $S^+_{n,k}$, and the containment is proper. Since $S^+_{n,k}$ is connected and $\alpha<1$, strict Perron--Frobenius monotonicity for $A_\alpha$ gives
$
        \rho_\alpha(S^+_{n,t})<\rho_\alpha(S^+_{n,k}).
$
If $G$ contained no $C_{2t+2}$, then by applying Theorem~\ref{thm:main} with parameter $t$, we have
$
        \rho_\alpha(G)\le \rho_\alpha(S^+_{n,t})<\rho_\alpha(S^+_{n,k}),
$
a contradiction.
Hence $G$ contains $C_{2t+2}$. Since $t$ was arbitrary, the result follows.
\end{proof}

\subsection{\texorpdfstring{A template switching principle for $A_\alpha$}{A template switching principle for A-alpha}}

Many exact spectral extremal results have the same final step. One first proves that every admissible graph is contained in, or can be switched into, a member of a small family of comparison templates; then Perron--Frobenius monotonicity shows that the maximum spectral radius is attained by one of the templates themselves. This mechanism appears in several recent spectral Tur\'an-type results, including results for color-critical graphs, disjoint color-critical graphs, wheels, fans and general spectral extremal theorems \cite{CioabaDesaiTait2023,LeiLi2024,ZhaoHuangLin2021,Zhang2025Fan,ByrneDesaiTait2025}. The following proposition isolates this final monotonicity step in the $A_\alpha$ setting. 

\begin{proposition}
\label{prop:aalpha-template-switching}
Let $F$ be a fixed graph, let $0\le\alpha<1$, and let $\mathcal T$ be a family of connected $F$-free graphs on the same vertex set. Define
$
        \rho_\alpha(\mathcal T)=\max_{T\in\mathcal T}\rho_\alpha(T).
$
Suppose that every $F$-free graph $G$ under consideration is a subgraph of some $T\in\mathcal T$. Then
$
        \rho_\alpha(G)\le \rho_\alpha(\mathcal T).
$
Moreover, if $G$ has maximum $A_\alpha$-spectral radius among all such $F$-free graphs and equality holds, then $G$ itself is a member of $\mathcal T$ attaining $\rho_\alpha(\mathcal T)$.
\end{proposition}

\begin{proof}
Choose $T\in\mathcal T$ such that $G\subseteq T$. Since $A_\alpha(G)\le A_\alpha(T)$ entrywise, Perron--Frobenius monotonicity gives
\[
        \rho_\alpha(G)\le \rho_\alpha(T)\le \rho_\alpha(\mathcal T).
\]
If equality holds and $G$ is a proper subgraph of the connected graph $T$, then strict Perron--Frobenius monotonicity for $A_\alpha$ with $\alpha<1$ gives
\[
        \rho_\alpha(G)<\rho_\alpha(T),
\]
a contradiction. Hence equality forces $G=T$ for some $T\in\mathcal T$. Since $\rho_\alpha(G)=\rho_\alpha(\mathcal T)$, this template $T$ attains $\rho_\alpha(\mathcal T)$.
\end{proof}
\subsection{Asymptotically sharp local-density estimates}

The next application gives an asymptotically sharp $A_\alpha$ estimate for
forbidden configurations that are detected inside neighborhoods.  The point is
that for $A_\alpha$, the correct first-order term is not always $n/2$.
There is a phase transition at $\alpha=1/2$: balanced complete bipartite graphs
give the lower bound $n/2+O(1)$, while stars give the lower bound
$\alpha n+O(1)$.  The following local-density theorem matches these two
constructions up to an additive constant depending only on the local-density
parameter.

\begin{lemma}
\label{lem:weighted-hereditary-sparsity-sharp}
Let $Q$ be a graph and let $c\ge0$. Suppose that for every $Y\subseteq V(Q)$, $e(Q[Y])\leq c|Y|$.
Then, for all nonnegative functions $f,g:V(Q)\to[0,\infty)$,
\[
        \sum_{u\in V(Q)}f(u)\sum_{v\sim_Q u}g(v)
        \le
        c\|g\|_\infty\sum_{u\in V(Q)}f(u)
        +
        c\|f\|_\infty\sum_{v\in V(Q)}g(v).
\]
\end{lemma}

\begin{proof}
If $\|f\|_\infty\|g\|_\infty=0$, there is nothing to prove. By homogeneity we
may assume $0\le f,g\le1$. For $0\le s,t\le1$, define
$
        S_s=\{u:f(u)\ge s\}$, and $
        T_t=\{v:g(v)\ge t\}.
$
For arbitrary $S,T\subseteq V(Q)$ we have
\[
        \sum_{v\in T}d_S(v)
        \le e(Q[S\cup T])+e(Q[S\cap T])
        \le c|S\cup T|+c|S\cap T|
        =c|S|+c|T|.
\]
Indeed, every edge counted by $\sum_{v\in T}d_S(v)$ lies in $Q[S\cup T]$,
and the only possible double-counting comes from edges inside $S\cap T$.
Integrating over the level sets gives
\[
\begin{aligned}
        \sum_{u}f(u)\sum_{v\sim_Q u}g(v)
        &=
        \int_0^1\int_0^1 \sum_{v\in T_t}d_{S_s}(v)\,dt\,ds \\
        &\le
        c\int_0^1|S_s|\,ds+c\int_0^1|T_t|\,dt \\
        &=
        c\sum_u f(u)+c\sum_v g(v).
\end{aligned}
\]
Rescaling $f$ and $g$ proves the general statement.
\end{proof}

\begin{lemma}
\label{lem:local-density-adj-signless}
Let $G$ be an $n$-vertex graph and let $c\ge0$. Suppose that, for every
$z\in V(G)$ and every $Y\subseteq N_G(z)$,
$
        e(G[Y])\le c|Y|.
$
Let $\lambda(G)$ be the adjacency spectral radius and let
$q(G)=\rho(D(G)+A(G))$ be the signless Laplacian spectral radius. Then
$
        \lambda(G)\le \frac n2+c$, 
        and $
        q(G)\le n+2c.
$
\end{lemma}

\begin{proof}
We first prove the adjacency estimate. If $\lambda(G)=0$, there is nothing to
prove. Otherwise let $x$ be a nonnegative Perron vector of $A(G)$ normalized by
$\max_v x_v=1$. Choose $z$ with $x_z=1$, and define
$
        W=N_G(z)$, $d=|W|$, and $ R=V(G)\setminus(W\cup\{z\}).$
The eigenvalue equation at $z$ gives
$        \sum_{u\in W}x_u=\lambda(G),
$
so $d\ge\lambda(G)$. Write $\lambda=\lambda(G)$. Applying the eigenvalue
equation once more at $z$ gives
$$
\begin{aligned}
        \lambda^2
        =
        \sum_{u\in W}\sum_{y\sim u}x_y  
        =
        d+
        \sum_{u\in W}\sum_{\substack{y\in W\\uy\in E(G)}}x_y
        +
        \sum_{u\in W}\sum_{\substack{y\in R\\uy\in E(G)}}x_y .
\end{aligned}
$$
The last term is at most $d|R|=d(n-d-1)$. Let $Q=G[W]$. By
Lemma~\ref{lem:weighted-hereditary-sparsity-sharp}, applied to $Q$ with
$f\equiv1$ and $g=x|_W$,
\[
        \sum_{u\in W}\sum_{\substack{y\in W\\uy\in E(G)}}x_y
        \le cd+c\sum_{y\in W}x_y
        =
        cd+c\lambda .
\]
Hence
\[
        \lambda^2
        \le
        d+d(n-d-1)+cd+c\lambda
        =
        d(n+c-d)+c\lambda .
\]
Thus $
        \lambda^2-c\lambda\le d(n+c-d).$
If $\lambda\le(n+c)/2$, then $\lambda\le n/2+c$. Otherwise
$\lambda>(n+c)/2$. Since $d\ge\lambda$ and the function
$t\mapsto t(n+c-t)$ is decreasing on $[(n+c)/2,\infty)$, we get
$
        d(n+c-d)\le \lambda(n+c-\lambda).
$
Therefore
$
        \lambda^2-c\lambda\le\lambda(n+c-\lambda),
$
and so $2\lambda\le n+2c$. This proves
$
        \lambda(G)\le \frac n2+c.
$

We now prove the signless Laplacian estimate. If $q(G)=0$, the claim is
trivial. Let $x$ be a nonnegative Perron vector of $Q(G)=D(G)+A(G)$,
normalized by $\max_v x_v=1$, and choose $z$ with $x_z=1$. Use the same
notation
$
        W=N_G(z)$, $ d=|W|$, and $R=V(G)\setminus(W\cup\{z\}).
$
Let
$
        S=\sum_{u\in W}x_u .
$
The signless eigenvalue equation at $z$ gives
$
        q(G)=d+S.
$
Write $q=q(G)$. Summing the signless eigenvalue equations over $u\in W$ gives
\[
        qS
        =
        \sum_{u\in W}d(u)x_u
        +
        \sum_{u\in W}\sum_{y\sim u}x_y .
\]
We split the right-hand side. The edge contributions through $z$ give $S+d$.
The internal contribution inside $W$ is
$
        2\sum_{u\in W}d_{G[W]}(u)x_u.
$
By Lemma~\ref{lem:weighted-hereditary-sparsity-sharp}, applied to $G[W]$ with
$f=x|_W$ and $g\equiv1$,
\[
        \sum_{u\in W}d_{G[W]}(u)x_u
        \le cS+cd.
\]
The contribution from edges between $W$ and $R$ is at most
$
        |R|S+d|R|=|R|(S+d),
$
because $d_R(u)\le |R|$ for each $u\in W$ and $x_y\le1$ for each $y\in R$.
Consequently,
\[
        qS
        \le
        S+d+2c(S+d)+|R|(S+d)
        =
        (S+d)(|R|+1+2c).
\]
Since $S+d=q$ and $|R|=n-d-1$, this gives
$
        qS\le q(n-d+2c).
$
As $q>0$, we obtain
$
        S\le n-d+2c.
$
Therefore
$
        q=d+S\le n+2c.
$
This proves the lemma.
\end{proof}

\begin{theorem}
\label{thm:sharp-local-density-aalpha}
Let $G$ be an $n$-vertex graph and let $c\ge0$. Suppose that, for every
$z\in V(G)$ and every $Y\subseteq N_G(z)$,
$
        e(G[Y])\le c|Y|.
$
Then, for every $0\le\alpha\le1$,
\[
        \rho_\alpha(G)
        \le
        \begin{cases}
        \displaystyle \frac n2+c,
        & 0\le\alpha\le \frac12,\\[6pt]
        \displaystyle
        (2\alpha-1)(n-1)+(1-\alpha)(n+2c),
        & \frac12\le\alpha\le1.
        \end{cases}
\]
In particular, for fixed $c$,
\[
        \rho_\alpha(G)\le \max\{\alpha,1/2\}n+O_c(1).
\]
Moreover, the estimate is asymptotically sharp.
\end{theorem}

\begin{proof}
Let $Q(G)=D(G)+A(G)$ and let $q(G)=\rho(Q(G))$. By
Lemma~\ref{lem:local-density-adj-signless}, we have that
$
        \lambda(G)\le \frac n2+c$ and $
        q(G)\le n+2c.$
If $0\le\alpha\le1/2$, then
$
        A_\alpha(G)
        =
        \alpha Q(G)+(1-2\alpha)A(G).
$
Since both coefficients are nonnegative and the matrices are symmetric,
\[
    \rho_\alpha(G)
        \le
        \alpha q(G)+(1-2\alpha)\lambda(G) 
        \le
        \alpha(n+2c)+(1-2\alpha)\left(\frac n2+c\right)
        =
        \frac n2+c.
\]
If $1/2\le\alpha\le1$, then
$
        A_\alpha(G)
        =
        (2\alpha-1)D(G)+(1-\alpha)Q(G).
$
Thus
\[
        \rho_\alpha(G)
        \le
        (2\alpha-1)\Delta(G)+(1-\alpha)q(G)
        \le
        (2\alpha-1)(n-1)+(1-\alpha)(n+2c).
\]
This proves the upper bound.

It remains to show asymptotic sharpness. The graph
$K_{\lfloor n/2\rfloor,\lfloor n/2\rfloor}$, with isolated vertices added if
necessary to make the order exactly $n$, satisfies the local condition with
$c=0$, because every neighborhood is independent. It is
$\lfloor n/2\rfloor$-regular on its nontrivial component, and hence
\[
        \rho_\alpha\bigl(K_{\lfloor n/2\rfloor,\lfloor n/2\rfloor}\bigr)
        =
        \lfloor n/2\rfloor
        =
        \frac n2+O(1).
\]
On the other hand, the star $K_{1,n-1}$ also satisfies the local condition with
$c=0$, and
\[
        \rho_\alpha(K_{1,n-1})
        \ge
        \alpha\Delta(K_{1,n-1})
        =
        \alpha(n-1).
\]
Therefore the supremum over this local-density class is at least
$
        \max\{\alpha,1/2\}n-O(1),
$
which matches the upper bound up to an additive constant depending only on
$c$.
\end{proof}

\begin{corollary}
\label{cor:sharp-local-density-forbidden}
Let $G$ be an $n$-vertex graph and let $0\le\alpha\le1$. Then the following
hold.

\begin{enumerate}[label=(\roman*)]
\item If $G$ is $(K_1\vee P_\ell)$-free, where $P_\ell$ denotes the path of
order $\ell$, then
\[
        \rho_\alpha(G)
        \le
        \begin{cases}
        \displaystyle \frac n2+\frac{\ell-2}{2},
        & 0\le\alpha\le \frac12,\\[6pt]
        \displaystyle
        (2\alpha-1)(n-1)+(1-\alpha)(n+\ell-2),
        & \frac12\le\alpha\le1.
        \end{cases}
\]
In particular, for fixed $\ell$,
\[
        \rho_\alpha(G)\le \max\{\alpha,1/2\}n+O_\ell(1),
\]
and this is asymptotically sharp.

\item If $G$ is $F_s$-free, where $F_s$ is the friendship graph consisting of
$s$ triangles sharing one common vertex, then
\[
        \rho_\alpha(G)
        \le
        \begin{cases}
        \displaystyle \frac n2+s-1,
        & 0\le\alpha\le \frac12,\\[6pt]
        \displaystyle
        (2\alpha-1)(n-1)+(1-\alpha)(n+2s-2),
        & \frac12\le\alpha\le1.
        \end{cases}
\]
In particular, for fixed $s$,
\[
        \rho_\alpha(G)\le \max\{\alpha,1/2\}n+O_s(1),
\]
and this is asymptotically sharp.
\end{enumerate}
\end{corollary}

\begin{proof}
For $(i)$, fix $z\in V(G)$. If $G[N_G(z)]$ contained a path of order $\ell$, then this path together with $z$ would form a copy of $K_1\vee P_\ell$. Therefore $G[N_G(z)]$ is $P_\ell$-free. The same is true for every induced subgraph $G[Y]$ with $Y\subseteq N_G(z)$. By the Erd\H{o}s--Gallai path theorem,
$
        e(G[Y])\le \frac{\ell-2}{2}|Y|.
$
Thus Theorem~\ref{thm:sharp-local-density-aalpha} applies with
$c=(\ell-2)/2$.

For $(ii)$, fix $z\in V(G)$. If $G[N_G(z)]$ contained a matching of size $s$, then those $s$ disjoint edges together with the apex $z$ would form a copy of $F_s$. Hence every induced subgraph $G[Y]$ with $Y\subseteq N_G(z)$ has matching number at most $s-1$.

We claim that every graph $Q$ on $m$ vertices with matching number at most
$s-1$ satisfies
$
        e(Q)\le (s-1)m.
$
Let $M$ be a maximum matching in $Q$, say $|M|=q\le s-1$. Put $U=V(M)$ and $W=V(Q)\setminus U$. Since $M$ is maximum, $W$ is independent. For each matched edge $ab\in M$, the number of edges from $\{a,b\}$ to $W$ is at most $|W|+1$. Indeed, if there were two distinct vertices $x,y\in W$ such that $ax$ and $by$ were edges, then replacing $ab$ by $ax$ and $by$ would produce a larger matching. Therefore,
$
        e(U,W)\le q(|W|+1).
$
Moreover, we have 
$
        e(Q[U])\le \binom{2q}{2}=q(2q-1).
$
Consequently,
\[
        e(Q)
        =
        e(Q[U])+e(U,W)
        \le
        q(2q-1)+q(|W|+1)
        =
        q(2q+|W|)
        =
        q|V(Q)|
        \le
        (s-1)|V(Q)|.
\]
Applying this to $Q=G[Y]$ gives
$
        e(G[Y])\le(s-1)|Y|.
$
Thus Theorem~\ref{thm:sharp-local-density-aalpha} applies with $c=s-1$.

For sharpness in both (i) and (ii), use the two constructions from Theorem~\ref{thm:sharp-local-density-aalpha}: the balanced complete bipartite graph gives $\rho_\alpha\ge n/2-O(1)$, and the star gives $\rho_\alpha\ge\alpha(n-1)$. Both graphs are triangle-free and have independent neighborhoods, so they are both $(K_1\vee P_\ell)$-free and $F_s$-free.
\end{proof}

\section*{Acknowledgements}

This work was supported by the National Key R\&D Program of China
(No.~2022YFA1006400) and the National Natural Science Foundation of China
(No.~12571376).

\section*{Declaration}
\noindent$\textbf{Conflict~of~interest}$
The authors declare that they have no known competing financial interests or personal relationships that could have appeared to influence the work reported in this paper.
	
\noindent$\textbf{Data~availability}$
Data sharing not applicable to this paper as no datasets were generated or analysed during the current study.

\medskip

\noindent\textbf{Use of generative AI tools.} During an initial exploratory stage, the authors used generative AI tools to assist in producing a preliminary proof for the adjacency-spectral problem. After independent verification, the AI-assisted argument established only the cubic upper bound $N(k)=O(k^3)$, based on a Perron-coordinate cutoff of order $1/k$ and switching estimates. The improvement to the linear range was developed by the authors. By replacing the cutoff-based estimates with a weighted rooted Erd\H{o}s--Gallai lemma, and combining it with a hub-deficit analysis, a completion inequality, comparison with $\rho_\alpha(S_{n,k})$, and $A_\alpha$-Rayleigh switching, the authors proved $N_\varepsilon(k)=\Theta_\varepsilon(k)$ for $0\le\alpha\le1-\varepsilon$. Generative AI tools were also used to help streamline the presentation of some proofs and polish the text. All AI-assisted material retained in the manuscript was independently checked and revised by the authors, who take full responsibility for the final content.

\bibliographystyle{abbrv}
\bibliography{alpha_even_cycle}

@article{Mantel1907,
  author  = {Mantel, W.},
  title   = {Problem 28},
  journal = {Wiskundige Opgaven met de Oplossingen},
  volume  = {10},
  pages   = {60--61},
  year    = {1907}
}

@article{Turan1941,
  author  = {Tur{\'a}n, Paul},
  title   = {On an extremal problem in graph theory},
  journal = {Mat. Fiz. Lapok},
  volume  = {48},
  pages   = {436--452},
  year    = {1941}
}

@mastersthesis{Nosal1970,
  author = {Nosal, Edward},
  title  = {Eigenvalues of Graphs},
  school = {University of Calgary},
  year   = {1970}
}

@article{Nikiforov2007,
  author  = {Nikiforov, Vladimir},
  title   = {Bounds on graph eigenvalues {II}},
  journal = {Linear Algebra Appl.},
  volume  = {427},
  number  = {2--3},
  pages   = {183--189},
  year    = {2007},
  doi     = {10.1016/j.laa.2007.06.019}
}

@article{Nikiforov2008,
  author  = {Nikiforov, Vladimir},
  title   = {A spectral condition for odd cycles in graphs},
  journal = {Linear Algebra Appl.},
  volume  = {428},
  number  = {7},
  pages   = {1492--1498},
  year    = {2008},
  doi     = {10.1016/j.laa.2007.10.018}
}

@article{Nikiforov2010,
  author  = {Nikiforov, Vladimir},
  title   = {The spectral radius of graphs without paths and cycles of specified length},
  journal = {Linear Algebra Appl.},
  volume  = {432},
  number  = {9},
  pages   = {2243--2256},
  year    = {2010},
  doi     = {10.1016/j.laa.2009.10.021}
}

@incollection{FurediSimonovits2013,
  author    = {F{\"u}redi, Zolt{\'a}n and Simonovits, Mikl{\'o}s},
  title     = {The history of degenerate (bipartite) extremal graph problems},
  booktitle = {Erd{\H{o}}s Centennial},
  series    = {Bolyai Soc. Math. Stud.},
  volume    = {25},
  pages     = {169--264},
  publisher = {J{\'a}nos Bolyai Math. Soc.},
  year      = {2013}
}

@article{Verstraete2000,
  author  = {Verstra{\"e}te, Jacques},
  title   = {On arithmetic progressions of cycle lengths in graphs},
  journal = {Combin. Probab. Comput.},
  volume  = {9},
  number  = {4},
  pages   = {369--373},
  year    = {2000},
  doi     = {10.1017/S0963548300004259}
}

@incollection{Verstraete2016,
  author    = {Verstra{\"e}te, Jacques},
  title     = {Extremal problems for cycles in graphs},
  booktitle = {Recent Trends in Combinatorics},
  series    = {IMA Vol. Math. Appl.},
  volume    = {159},
  pages     = {83--116},
  publisher = {Springer},
  address   = {Cham},
  year      = {2016},
  doi       = {10.1007/978-3-319-24298-9_4}
}

@article{CioabaDesaiTait2024,
  author  = {Cioab{\u a}, Sebastian M. and Desai, Dheer Noal and Tait, Michael},
  title   = {The spectral even cycle problem},
  journal = {Combinatorial Theory},
  volume  = {4},
  number  = {1},
  pages   = {Paper No. 10},
  year    = {2024},
  doi     = {10.5070/C64163847},
}

@article{ZhaiLin2020,
  author  = {Zhai, Mingqing and Lin, Huiqiu},
  title   = {Spectral extrema of graphs: Forbidden hexagon},
  journal = {Discrete Math.},
  volume  = {343},
  number  = {10},
  pages   = {112028},
  year    = {2020},
  doi     = {10.1016/j.disc.2020.112028}
}

@article{ZhaiWangFang2020,
  author  = {Zhai, Mingqing and Wang, Bing and Fang, Longfei},
  title   = {The spectral {Tur\'{a}n} problem about graphs with no $6$-cycle},
  journal = {Linear Algebra Appl.},
  volume  = {590},
  pages   = {22--31},
  year    = {2020},
  doi     = {10.1016/j.laa.2019.12.032}
}

@article{WangKangXue2023,
  author  = {Wang, Jing and Kang, Liying and Xue, Yisai},
  title   = {On a conjecture of spectral extremal problems},
  journal = {J. Combin. Theory Ser. B},
  volume  = {159},
  pages   = {20--41},
  year    = {2023},
  doi     = {10.1016/j.jctb.2022.11.002}
}

@article{Woodall1976,
  author  = {Woodall, D. R.},
  title   = {Maximal circuits of graphs. {I}},
  journal = {Acta Math. Acad. Sci. Hungar.},
  volume  = {28},
  number  = {1--2},
  pages   = {77--80},
  year    = {1976},
  doi     = {10.1007/BF01902497}
}

@article{LiNing2023,
  author  = {Li, Binlong and Ning, Bo},
  title   = {Stability of {Woodall's} theorem and spectral conditions for large cycles},
  journal = {Electron. J. Combin.},
  volume  = {30},
  number  = {1},
  pages   = {Paper No. 1.39},
  year    = {2023},
  doi     = {10.37236/11641}
}

@article{Zhang2024ConsecutiveCycles,
  author  = {Zhang, Wenqian},
  title   = {The spectral radius, maximum average degree and cycles of consecutive lengths of graphs},
  journal = {Graphs Combin.},
  volume  = {40},
  number  = {2},
  pages   = {Paper No. 32},
  year    = {2024},
  doi     = {10.1007/s00373-024-02761-0}
}

@article{Zhang2025Skeletons,
  author       = {Zhang, Wenqian},
  title        = {Spectral skeletons and applications},
  journal      = {arXiv preprint arXiv:2501.14218},
  year         = {2025},
  archivePrefix = {arXiv},
  eprint       = {2501.14218},
  primaryClass = {math.CO},
}

@article{CioabaDesaiTait2023,
  author  = {Cioab{\u a}, Sebastian M. and Desai, Dheer Noal and Tait, Michael},
  title   = {A spectral {Erd\H{o}s--S\'{o}s} theorem},
  journal = {SIAM J. Discrete Math.},
  volume  = {37},
  number  = {3},
  pages   = {2228--2239},
  year    = {2023},
  doi     = {10.1137/22M150650X}
}

@article{LeiLi2024,
  author  = {Lei, Xingyu and Li, Shuchao},
  title   = {Spectral extremal problem on disjoint color-critical graphs},
  journal = {Electron. J. Combin.},
  volume  = {31},
  number  = {1},
  pages   = {Paper No. 1.25},
  year    = {2024},
  doi     = {10.37236/12245}
}

@article{ZhaoHuangLin2021,
  author  = {Zhao, Yanhua and Huang, Xueyi and Lin, Huiqiu},
  title   = {The maximum spectral radius of wheel-free graphs},
  journal = {Discrete Math.},
  volume  = {344},
  number  = {5},
  pages   = {112341},
  year    = {2021},
  doi     = {10.1016/j.disc.2021.112341}
}

@article{Zhang2025Fan,
  author       = {Zhang, Wenqian},
  title        = {Spectral extrema of graphs forbidding a fan},
  journal      = {arXiv preprint arXiv:2508.05911},
  year         = {2025},
  archivePrefix = {arXiv},
  eprint       = {2508.05911},
  primaryClass = {math.CO},
}

@article{Nikiforov2017,
  author  = {Nikiforov, Vladimir},
  title   = {Merging the {$A$}- and {$Q$}-spectral theories},
  journal = {Appl. Anal. Discrete Math.},
  volume  = {11},
  number  = {1},
  pages   = {81--107},
  year    = {2017},
  doi     = {10.2298/AADM1701081N}
}

@article{ErdosGallai1959,
  author  = {Erd{\H{o}}s, Paul and Gallai, Tibor},
  title   = {On maximal paths and circuits of graphs},
  journal = {Acta Math. Acad. Sci. Hungar.},
  volume  = {10},
  pages   = {337--356},
  year    = {1959},
  doi     = {10.1007/BF02024498}
}

@article{VossZuluaga1977,
  author  = {Voss, H. J. and Zuluaga, C.},
  title   = {Maximale gerade und ungerade Kreise in Graphen, {I}},
  journal = {Wissenschaftliche Zeitschrift der Technischen Hochschule Ilmenau},
  volume  = {23},
  number  = {4},
  pages   = {57--70},
  year    = {1977}
}

@article{liu2023unsolved,
  title={Unsolved problems in spectral graph theory},
  author={Liu, Lele and Ning, Bo},
  journal={arXiv preprint arXiv:2305.10290},
  year={2023}
}

@article{Wilf1986,
  author  = {Wilf, Herbert S.},
  title   = {Spectral bounds for the clique and independence numbers of graphs},
  journal = {J. Combin. Theory Ser. B},
  volume  = {40},
  number  = {1},
  pages   = {113--117},
  year    = {1986},
  doi     = {10.1016/0095-8956(86)90069-9}
}

@article{Stanley1987,
  author  = {Stanley, Richard P.},
  title   = {A bound on the spectral radius of graphs with $e$ edges},
  journal = {Linear Algebra Appl.},
  volume  = {87},
  pages   = {267--269},
  year    = {1987},
  doi     = {10.1016/0024-3795(87)90163-2}
}

@article{BabaiGuiduli2009,
  author  = {Babai, L{\'a}szl{\'o} and Guiduli, Barry},
  title   = {Spectral extrema for graphs: the {Zarankiewicz} problem},
  journal = {Electron. J. Combin.},
  volume  = {16},
  number  = {1},
  pages   = {Research Paper 123},
  year    = {2009},
  doi     = {10.37236/212}
}

@misc{ByrneDesaiTait2025,
  author        = {Byrne, John and Desai, Dheer Noal and Tait, Michael},
  title         = {A general theorem in spectral extremal graph theory},
  year          = {2025},
  eprint        = {2401.07266},
  archivePrefix = {arXiv},
  primaryClass  = {math.CO},
  howpublished  = {To appear in Trans. Amer. Math. Soc.}
}

@article{LiYu2023AalphaEvenCycles,
  author  = {Li, Shuchao and Yu, Yuantian},
  title   = {On {$A_\alpha$} spectral extrema of graphs forbidding even cycles},
  journal = {Linear Algebra Appl.},
  volume  = {668},
  pages   = {11--27},
  year    = {2023},
  doi     = {10.1016/j.laa.2023.03.018}
}

@article{ChenLiLiYuZhang2023AalphaES,
  author  = {Chen, Ming-Zhu and Li, Shuchao and Li, Zhao-Ming and Yu, Yuantian and Zhang, Xiao-Dong},
  title   = {An {$A_{\alpha}$}-spectral {Erd\H{o}s--S\'{o}s} theorem},
  journal = {Electron. J. Combin.},
  volume  = {30},
  number  = {3},
  pages   = {Paper No. 3.34},
  year    = {2023},
  doi     = {10.37236/11593}
}

@article{ChenLiuZhang2023LinearForestsAalpha,
  author  = {Chen, Ming-Zhu and Liu, A-Ming and Zhang, Xiao-Dong},
  title   = {On the {$A_\alpha$}-spectral radius of graphs without linear forests},
  journal = {Appl. Math. Comput.},
  volume  = {450},
  pages   = {128005},
  year    = {2023},
  doi     = {10.1016/j.amc.2023.128005}
}

@article{LiYuZhang2023AalphaErdosPosa,
  author  = {Li, Shuchao and Yu, Yuantian and Zhang, Huihui},
  title   = {An {$A_\alpha$}-spectral {Erd\H{o}s--P\'{o}sa} theorem},
  journal = {Discrete Math.},
  volume  = {346},
  number  = {9},
  pages   = {113494},
  year    = {2023},
  doi     = {10.1016/j.disc.2023.113494}
}

@article{NikiforovYuan2015QEvenCycles,
  author  = {Nikiforov, Vladimir and Yuan, Xiying},
  title   = {Maxima of the {$Q$}-index: forbidden even cycles},
  journal = {Linear Algebra Appl.},
  volume  = {471},
  pages   = {636--653},
  year    = {2015},
  doi     = {10.1016/j.laa.2015.01.032}
}

@article{LiNingConsecutive2023,
  author  = {Li, Binlong and Ning, Bo},
  title   = {Eigenvalues and cycles of consecutive lengths},
  journal = {J. Graph Theory},
  volume  = {103},
  number  = {3},
  pages   = {486--492},
  year    = {2023},
  doi     = {10.1002/jgt.22930}
}

@article{TaitTobin2017,
  author  = {Tait, Michael and Tobin, Josh},
  title   = {Three conjectures in extremal spectral graph theory},
  journal = {J. Combin. Theory Ser. B},
  volume  = {126},
  pages   = {137--161},
  year    = {2017},
  doi     = {10.1016/j.jctb.2017.04.006}
}

@article{CioabaDesaiTait2022OddWheels,
  author  = {Cioab{\u{a}}, Sebastian M. and Desai, Dheer Noal and Tait, Michael},
  title   = {The spectral radius of graphs with no odd wheels},
  journal = {European J. Combin.},
  volume  = {99},
  pages   = {103420},
  year    = {2022},
  doi     = {10.1016/j.ejc.2021.103420}
}

@article{DvorakMohar2010,
  author  = {Dvo{\v r}{\'a}k, Zden{\v e}k and Mohar, Bojan},
  title   = {Spectral radius of finite and infinite planar graphs and of graphs of bounded genus},
  journal = {Journal of Combinatorial Theory, Series B},
  volume  = {100},
  number  = {6},
  pages   = {729--739},
  year    = {2010},
  doi     = {10.1016/j.jctb.2010.07.006}
}

@article{Tait2019CdV,
  author  = {Tait, Michael},
  title   = {The {C}olin de {V}erdi{\`e}re parameter, excluded minors, and the spectral radius},
  journal = {Journal of Combinatorial Theory, Series A},
  volume  = {166},
  pages   = {42--58},
  year    = {2019},
  doi     = {10.1016/j.jcta.2019.02.018}
}

\end{document}